\documentclass[english,11pt]{article}
\usepackage{geometry}
\usepackage{float}
\usepackage{pgf,tikz}
\usepackage{mathrsfs}
\usetikzlibrary{arrows}
\usepackage{fullpage}
\usepackage{bm}
\usepackage{comment}
\usepackage{fancyhdr}
\usepackage{caption}
\usepackage{epsfig}
\usepackage{amsthm}
\usepackage{amsfonts}
\usepackage{enumerate}
\usepackage{subfigure}
\usepackage{setspace}
\usepackage{showlabels}
\usepackage{mathtools}
\usepackage{upgreek}
\usepackage[linesnumbered,ruled,vlined]{algorithm2e}
\usepackage{tikz-3dplot} 
\usepackage{pgfplots}
\usepackage{amsmath,amssymb}
\usepackage{color}% Include colors for document elements
\usepackage{dcolumn}% Align table columns on decimal point
\usepackage{bm}% bold math
\usepackage{pgfplots}
\pgfplotsset{compat=1.11}
\usepackage{etoolbox}
\newcounter{dummy}
\usepackage{enumitem} 
\makeatletter  
\newcommand\myitem[1][]{\item[#1]\refstepcounter{dummy}\def\@currentlabel{#1}}   
\makeatother 

\usepackage{lineno}
\usetikzlibrary{matrix}

\newtheorem{theorem}{Theorem}[section]

\newtheorem{lemma}[theorem]{Lemma}
\newtheorem{proposition}[theorem]{Proposition}
\newtheorem*{remark}{Remark}

\newcommand{\lambert}{L}        % Lambert map
\newcommand{\R}{\mathbb{R}}     % Real numbers
\newcommand{\HH}{\mathbb{H}} % Hyperbolic space
\newcommand{\PT}{\mathrm{PT}}   % Parallel transport
\DeclareMathOperator{\vol}{vol} % Volume
\newcommand{\dvol}{\mathrm{dvol}} % Density of the Riemannian volume.
\newcommand{\mmid}{\hspace{1pt}|\hspace{1pt}}

\theoremstyle{definition}
\newtheorem{definition}[theorem]{Definition}
\numberwithin{equation}{section}

\usepackage{tikz-cd,mathtools}
\usepackage{pgfplots}
\pgfplotsset{width=7cm,compat=1.8}

\usepackage{mleftright}
\usepackage{xparse}

\usepackage{amsmath}

\usepackage[utf8]{inputenc}

\usepackage[final]{hyperref} % adds hyper links inside the generated pdf file
\usepackage{setspace}

\usepackage{txfonts}

\DeclareUnicodeCharacter{0301}{\'{e}}

\usepackage[]{biblatex}
 
\hypersetup{
	colorlinks=true,       % false: boxed links; true: colored 
	linkcolor=blue,        % color of internal links
	citecolor=blue,        % color of links to bibliography
	filecolor=magenta,     % color of file links
	urlcolor=blue         
}
\usepackage{breqn}

\SetKwInput{KwInput}{Input}
\SetKwInput{KwOutput}{Output}
\SetKwInput{KwRequire}{Require}

\NewDocumentCommand{\evalat}{sO{\big}mm}{%
  \IfBooleanTF{#1}
   {\mleft. #3 \mright|_{#4}}
   {#3#2|_{#4}}%
}

\renewcommand{\openbox}{\leavevmode
  \hbox to.77778em{%
  \hfil\vrule
  \vbox to.675em{\hrule width.6em\vfil\hrule}%
  \vrule\hfil}}

\title{Wrapped Distributions on homogeneous Riemannian manifolds}
\author{Fernando Galaz-Garc\'ia \thanks{Department of Mathematical Sciences, Durham University, UK, fernando.galaz-garcia@durham.ac.uk}, Marios Papamichalis \thanks{Fox School of Business, Temple University, USA,tun47241@temple.edu},Kathryn Turnbull \thanks{Fox School of Business, Temple University, USA, tun30719@temple.edu} \\ Sim\'on Lunag\'omez \thanks{Departamento de Estadistica, Instituto Tecnologico Autonomo de Mexico}, Edoardo Airoldi \thanks{Fox School of Business, Temple University, USA , airoldi@temple.edu} }

\begin{document}

\maketitle

\begin{abstract}
We provide a general  framework for constructing probability distributions on Riemannian manifolds, taking advantage of area-preserving maps and isometries. Control over distributions' properties, such as parameters, symmetry and modality yield a family of flexible distributions that are straightforward to sample from, suitable for use within Monte Carlo algorithms and latent variable models, such as autoencoders. As an illustration, we empirically validate our approach by utilizing our proposed distributions within a variational autoencoder and a latent space network model. Finally, we take advantage of the generalized description of this framework to posit questions for future work.
\end{abstract}


\footnotetext{
Dr. Fernando Galaz-Garc\'ia is the main author of the current paper. All theoretical/mathematical proofs belong to him. Dr. Marios Papamichalis and Dr. Kathryn Turnbull worked their initial ideas, that came as a consequence of their first paper \cite{Papamichalis2021}, with him. They both worked under the supervision and guidance of Prof. Sim\'on Lunag\'omez and Prof. Edoardo Airoldi for the statistical and machine learning part of the paper. This is the first version of the current paper. Code, corrections and more figures will be available after submission.}

%\newpage

% MATH SUBJECT CLASSIFICATION
%\subjclass[2010]{53C20}

% KEYWORDS

% ABSTRACT
\newpage

\setcounter{tocdepth}{1}
%\tableofcontents

\begin{comment}

\section*{To do}

\begin{enumerate}
    \item S3: Write map $\lambda$ for general curvature.
    \item S3: Restructure S3: 3.1 Wrapping by diffeomoprhisms. Then distributions on the hyperboloid and on the sphere. 3.2: Wrapping by volume preserving maps. Then distributions on the hyperboloid and on the sphere. That is, first have the general results. Then illustrate them with the hyperboloid and sphere cases.
    \item S4: Keep on trucking.
    \item Appendix: refine and link to S4
    \item Revisit Mathieu etc to make sure we have put our work in context appropriately
\end{enumerate}

\end{comment}

\section{Introduction}

Probability distributions play a fundamental role in statistical data analysis where, for continuous data, the dominant assumption is to consider random variables in Euclidean space. However, the Euclidean assumption is not appropriate for some data types and this motivates the development and study of distributions in non-Euclidean spaces. Notable examples include directional statistics (see \cite{mardia1999_chapter9}), in which observations typically lie on a sphere, and data that are expressed as tensors, such as covariance matrices (\cite{pennec2006riemannian, said2017gaussian, said2017riemannian}) and data structures which arise in image and signal processing applications (see \cite{arnaudon2013riemannian, barachant2011multiclass}). Furthermore, in recent years, latent variable models have been shown to offer superior performance when the parameters are modelled in non-Euclidean spaces. Variational autoencoders (see \cite{kingma2019, nagano2019wrapped, skopek2019mixed, mathieu2019,ovinnikov2019}) and latent space network models (see \cite{hoff2002latent, lubold2020identifying, smith2019geometry, mccormick2015,Papamichalis2021}) offer two pertinent examples where it is most typical for non-Euclidean latent variables to be modelled in spherical or hyperbolic space. In this context, these choices of underlying geometry have been shown to exhibit desirable properties including tree structures (see \cite{nickel2017poincare}) and networks with particular interaction patterns, such as power-law degree distributions (see \cite{krioukov2010}).\\

In this article we consider distributions on homogeneous Riemannian manifolds, namely 
%spaces which locally look Euclidean,
%differentiable manifolds on which we can measure lengths and angles between tangent vectors  differentiably. 
Riemannian manifolds equipped with a transitive isometric action of a Lie group (see, for example, \cite{lee2006}), 
with a particular focus on spherical and hyperbolic geometries. There exists an increasingly rich literature on this topic which can broadly be divided into intrinsic and extrinsic methodology where, in the latter case, a manifold is viewed as embedded within a larger Euclidean ambient space. For spherical data, a plethora of, primarily extrinsic, Gaussian distributions have been proposed from directional statistics (see \cite{mardia1999_chapter9, hauberg2018directional}) which includes the well-known von Mises-Fisher (see \cite{vonmises1918, fisher1953dispersion}) and Kent (see \cite{kent1982}) distributions. For hyperbolic geometry, both intrinsic and extrinsic distributions have been proposed and these include the maximum entropy Normal (see \cite{said2014new,pennec2006intrinsic}) and wrapped distributions in which a Euclidean distribution is viewed as part of the tangent space and projected onto the manifold (see \cite{grattarola2019adversarial}, \cite{nagano2019wrapped}). More generally, \cite{pennec2006intrinsic} provide a general definition of probability measures on Riemannian manifolds and a summary of distributions in hyperbolic space is given in the appendix of \cite{mathieu2019}.\\

\begin{comment}

\end{comment}\textcolor{red}{[K:Are we missing anything? F: Maybe also paper by Battacharya and Patrangenaru? It is more theoretical but it may be worth mentioning it.]} \\

\textcolor{red}{Add clear description of the novelty of our contributions.} \\
\end{comment}

In contrast with \cite{nagano2019wrapped,mathieu2019,skopek2019mixed}, the family presented in this paper does not distort the distribution and associates the mean and variance parameters of the distribution in the manifold directly with initial mean and variance of the distribution in the tangent plane. More specifically, the Wrapped normal Gaussian Distribution with Lambert mapping preserves the measure, which means that the properties, i.e. symmetry, unimodality, expectations and variances are preserved as well, making the statistical analysis easy to handle by offering better control over the parameters. On the other hand, exponential mapping which is used as a diffeomorphism in \cite{nagano2019wrapped,mathieu2019,skopek2019mixed}, preserves the distance but distorts the distribution, leading to a more complex and not in terms of parameters not insightful formula. In short, the practitioner is provided with a distribution which is equipped with explicit, simpler, fast and easy to calculate formulas. Moreover, as mentioned above the advantages of this parametric, non-distorted distribution is that it gives you real information regarding the population which can be controlled by its parameters. For example, in terms of variance, this parametric distribution can perform quite well when they have spread over when data happens to be different. Another benefit of this type of parametric distributions, again due to the association with the normal includes statistical power which means that it has more power than other tests. Therefore, you will be able to find an effect that is significant when one will exist truly.\\

 \begin{comment} 
 
\begin{enumerate}
\item Our contributions: generalise this procedure, explore different mappings between tangent and manifold, framework to highlight questions for future work
\end{enumerate}

In this article we take advantage of non-euclidean smooth geometric spaces, \emph{Riemannian manifolds}, and their properties to construct a procedure to define new probability Guassian-like Riemannian distributions. We distinguish our approach in three different settings. Our approach is based on isomerties of the spaces and  consistent, thorough and universal for all smooth symmetric manifolds. The usefulness of these distributions stems from the fact that they are completely tractable, analytically or numerically and can lead to efficient statistical learning algorithms. Explicitly, we construct distributions on a given manifold that one can do computations with, using isometries in homogeneous spaces and area preserving maps. Other approaches can be described as cases where a different, not measure preserving map, is used (e.g exponential map). This fact leads to more complicated distributions which is difficult to extract analytical results. Furthermore, we use the method in the present article to provide rigorous constructions for latent space models on networks and variational autoencoders in deep learning whose underlying geometry is that of a Riemannian manifold. To this end, we provide a solid starting point for future work, which involves orbifold, non smooth manifolds and connections between topology and statistical literature.  \\ 

\end{comment}

By building on the construction of \cite{nagano2019wrapped}, we consider a procedure for generating wrapped probability distributions in homogeneous Riemannian manifolds which includes the sphere and hyperbolic geometry as special cases. This construction is computationally convenient and straightforward to sample from, making distributions derived in this manner particular convenient for practitioners to be used in modelling and as part of Monte Carlo sampling algorithms. Using properties of homogeneous spaces, we simplify the procedure outlined in \cite{nagano2019wrapped} for hyperbolic geometry and further describe an equivalent construction in spherical geometry similar to that presented in \cite{skopek2019mixed}. Furthermore, we present a generalised description of this procedure and show how this can be used to gain insight on properties of these distributions and highlight directions for future work. We also consider alternative choices for mapping between the tangent plane and the manifold, and demonstrate the applicability of these distributions as part of variational autoencoders and latent space network models.\\

\begin{comment}

\begin{itemize}
    \item Organisation of this article
\end{itemize}

\end{comment}

Our article is organized as follows. In Section \ref{sec:background} we provide 
rigorous definitions of Riemannian geometry tools that we will use in subsequent sections. 
In Section \ref{sec:wrapped_prob_distns} we focus on the construction of a distribution on a smooth manifold surface. In Section \ref{sec:theory}, we provide theoretical results which enable the statistical learning of
the probabilistic distribution on Riemannian manifolds that
could never have been considered before and thorough analytical results that concern the properties of the distribution. Experimental studies and results that focuses on the applicability of the previous section like latent space models on networks and variational autoencoders in which those distributions could be practical are presented in Section \ref{sec:experiments}. Then, in Section \ref{sec:discussion} we conclude the paper.

% SECTION: RIEMANNIAN GEOMETRY
%------------------------------

    \section{Background: Riemannian Geometry} 
\label{sec:background}

Riemannian geometry provides a toolkit for studying the geometry of \emph{smooth manifolds}. Intuitively, these are spaces which locally look like $n$-dimensional Euclidean space $\mathbb{R}^n$ in such a way that we can do differential calculus on them. Each point $p$ in a smooth manifold $M$ carries a copy of $\mathbb{R}^n$, the \emph{tangent space} to $M$ at $p$, denoted by $T_pM$. A \emph{Riemannian metric} $g = \{g_p\}_{p\in M}$ on $M$ is a collection of inner products $g_p$ on each tangent space $T_pM$ that varies differentiably with $p$. The pair $(M,g)$ is a \emph{Riemannian manifold}. The Riemannian metric allows one to measure angles between tangent vectors, as well as lengths of piecewise differentiable curves in $M$. This, in turn, induces a distance function $d$ on $M$, making the pair $(M,d)$ into a metric space. Further fundamental invariants involve different notions of curvature, such as \emph{sectional}, \emph{Ricci}, and \emph{scalar} curvature, of which sectional curvature is the most geometric one and generalises the Gaussian curvature for surfaces. In this article, when we refer to \textit{curvature}, we will always mean sectional curvature. Here, we briefly review the basic concepts from Riemannian geometry that we will use in subsequent sections. For further details we refer the reader to the many available textbooks on Riemannian geometry, such as  \cite{doCarmo1992,lee2006,petersen2016,sakai1998}. We will use  \cite{doCarmo1992}, \cite{lee2006}, and \cite{sakai1998} as our main references.

\subsection{Smooth manifolds}

% DEF: TOPOLOGY
% cf. Bredon, Ch. 1, Def. 2.8.

% \begin{definition}[Topological space]
% A \emph{topological space} is a set $X$ together with a collection $\mathcal{T}$ of subsets of $X$, called \emph{open} sets, which satisfy the following conditions:
% \begin{enumerate}
% \item Both $X$ and the empty set $\varnothing$ are open sets.
% \item The union of an arbitrary collection of open sets is an open set.
% \item The intersection of any two open sets  is an open set.
% \end{enumerate}
% A subset of $X$ is \textit{closed} if its complement is open. The collection $\mathcal{T}$ of open sets of $X$ is called a \textit{ topology}. A collection $\mathcal{B}$ of subsets of $X$ is a \textit{basis} for the topology $\mathcal{T}$ if every open set is the union of members of $\mathcal{B}$. We say that $X$ is \textit{second countable} if its topology has a countable basis. 
% \end{definition}

%For ease of notation, we will avoid  mentioning explicitly the topology of a given topological space. 
%A topological space $X$ is \emph{Hausdorff} if, for every pair of points $p,q \in X$, there exist disjoint open subsets $U,V\subset X$ such that $p\in U$ and $q\in V$. A subset $N\subset X$ is a \textit{neighbourhood} of a point $p$ in $X$ if there is an open set $U\subset X$ with $p\in U$. A function $f\colon X\to Y$ between topological spaces is \emph{continuous} if the preimage of any open set in $X$ is open in $Y$. A \textit{map} is a continuous function.

Recall that a topological space $X$ is \emph{Hausdorff} if, for every pair of points $p,q \in X$, there exist disjoint open subsets $U,V\subset X$ such that $p\in U$ and $q\in V$. A subset $N\subset X$ is a \textit{neighbourhood} of a point $p$ in $X$ if there is an open set $U\subset N$ with $p\in U$. 
We say that $X$ is \textit{second countable} if its topology has a countable basis. A \textit{map} is a continuous function.

% DEF: DIFFERENTIABLE MANIFOLD
% cf. Bredon Ch. II, Def. 2.1
% do Carmo, Ch. 0, Def. 2.1

\begin{definition}[$n$-dimensional differentiable manifold]
An \textit{$n$-dimensional differentiable}  \textit{manifold} is a second countable Hausdorff topological space $M$ satisfying the following conditions:
\begin{enumerate}
    \myitem[(C1)] \label{item:c1} For each point $p\in M$, there exists an open set $U$ containing $p$, an open set $\tilde{U}\subset \R^n$, and a homeomorphism $\varphi\colon U\subset M\to \tilde{U}\subset \R^n$ (i.e. a continuous bijective map with continuous inverse). The pair $(U,\varphi)$ is called a \textit{coordinate chart}.
    \myitem[(C2)] \label{item:c2} For any two coordinate charts $(U,\varphi)$, $(V,\psi)$ with $U\cap V \neq \varnothing$, the \textit{change of coordinates map} $\psi\circ\varphi^{-1}\colon \varphi(U\cap V)\to\psi(U\cap V)$ is a $C^\infty$   map between open subsets of $\R^n$.
    \myitem[(C3)] \label{item:c3} The family $\{(U,\varphi)\}$ is maximal with respect to conditions \ref{item:c1} and \ref{item:c2}, i.e., if $\{V,\psi\}$ is a family of charts satisfying conditions C1 and C2, then this family it must be contained in $\{(U,\varphi)\}$. %\textcolor{brown}{We call this family an \textit{atlas} for $M$.}
    
\end{enumerate}
\end{definition}

Condition \ref{item:c1} formalises the notion that a manifold is locally Euclidean, i.e.\ that within some neighbourhood, an $n$-dimensional manifold ``looks like'' $\R^n$.
Note that, given a coordinate chart $(U,\varphi)$ of a differentiable manifold $M$, we may introduce (Euclidean) local coordinates on $U$ by considering  $\varphi(p) = (x_1(p), x_2(p), \dots, x_n(p))\in \R^n$ for any $p\in U\subset M$. Condition \ref{item:c2} allows us to formalise differentiable objects on $M$. %Indeed, we say that a map $f\colon M\to N$ between smooth manifolds of dimension $m$ and $n$, respectively, is  \textit{differentiable} at $p\in M$ if, given a chart $(V,\psi)$ of $N$ at $f(p)$, there exists a chart $(U,\varphi)$ of $M$ at $p$ such that $f(U)\subset V$ and the map $\psi\circ f\circ\varphi^{-1}\colon \tilde{U}\subset \R^m\to \R^n$ is $C^\infty$  at $\varphi(p)$.  The map $f\colon M\to N$ is \textit{differentiable} if it is differentiable at all $p\in M$. A \textit{diffeomorphism} between differentiable manifolds is a differentiable bijective map with differentiable inverse. 

\begin{definition}[Differentiable function, diffeomorphism]
A map $f\colon M\to N$ between smooth manifolds of dimension $m$ and $n$, respectively, is  \textit{differentiable} at $p\in M$ if, given a chart $(V,\psi)$ of $N$ at $f(p)$, there exists a chart $(U,\varphi)$ of $M$ at $p$ such that $f(U)\subset V$ and the map $\psi\circ f\circ\varphi^{-1}\colon \tilde{U}\subset \R^m\to \R^n$ is $C^\infty$  at $\varphi(p)$.  The map $f\colon M\to N$ is \textit{differentiable} if it is differentiable at all $p\in M$. A \textit{diffeomorphism} between differentiable manifolds is a differentiable bijective map with differentiable inverse. 
\end{definition}

From now on we will assume that all our manifolds are connected, i.e. they cannot be divided into disjoint nonempty open sets.\\

With the definition of a differentiable manifold in hand, we now want to formalise a notion of distance on a given differentiable manifold. We will do this by means of a Riemannian metric. To define this object, we first need to define the tangent space to a manifold $M$ at a point $p \in M$. Informally, the tangent space to $M$ at a point $p\in M$ consists of all tangent directions of smooth curves in $M$ passing through $p$.

% DEF: TANGENT PLANE
% cf. do Carmo Ch. 0, def. 2.6.

\begin{definition}[Tangent vector, tangent space]
  Let $M$ be a differentiable manifold and fix $\varepsilon>0$. A differentiable function $\gamma\colon (-\varepsilon, \varepsilon)\to M$ is a (differentible) \textit{curve} in $M$. Fix $p\in M$, suppose that $\gamma(0)=p$, and let $\mathcal{F}(p)$ be the set of real-valued functions on $M$ that are differentiable at $p$. The \textit{tangent vector to the curve $\gamma$} at $t=0$ is a function $\gamma'(0)\colon \mathcal{F}(p)\to \R$ given by
  \[
  \gamma'(0)(f) = \evalat[\Bigg]{\frac{d(f\circ\gamma)}{dt}}{t=0}, \quad f\in \mathcal{F}(p).
  \]
  A \textit{tangent vector at} $p$ is the tangent vector at $t=0$ of some curve $\gamma\colon(-\varepsilon, \varepsilon)$ with $\gamma(0) = p$. The \textit{tangent space to $M$ at $p$}, denoted by $T_pM$, is the set of all tangent vectors to $M$ at $p$.
\end{definition}
If $M$ is an $n$-dimensional differentiable manifold, each tangent space $T_pM$, for $p\in M$, is a real vector space isomorphic to $\R^n$. Each differentiable function $f\colon M\to N$ between smooth manifolds induces, for each point $p\in M$, a linear mapping $df_p\colon T_pM \to T_{f(p)}N$, called the \textit{differential} of $f$ at $p$. Note that, if $f$ is a diffeomorphism, then its differential $df_p$ at any point $p\in M$ is a linear isomorphism between vector spaces.
%\textcolor{brown}{We say that a differentiable manifold is \textit{orientable} if it admits an atlas such that the differential of every transition map has positive determinant.} 

A  \textit{vector field} $X$ on $M$ is a correspondence that associates to each point $p\in M$ a tangent vector $X(p)\in T_pM$. Given a coordinate chart $\varphi\colon U\subset M\to \tilde{U}\subset \R^n$ around $p\in M$, with $\varphi = (x_1,\ldots,x_n)$, we may write the value of $X$ at $p$ as 
\begin{align}
\label{eq:vector_field_coordinates}
X(p) = \sum_{i=1}^n a^i(p)\frac{\partial}{\partial x_i}(p),
\end{align}
where each $a_i\colon U\subset M\to \R$ is a real-valued function and $\frac{\partial}{\partial x_i}(p)=d\varphi^{-1}_{\varphi(p)}e_i$ is the image of the $i$-th element of the canonical basis of $\R^n$ under the differential $d\varphi^{-1}_{\varphi(p)}\colon T_{\varphi(p)}\R^n\cong\R^n \to T_pM\cong\R^n$. Note that the tangent vectors $\left\{\frac{\partial}{\partial x_i}(p)\right\}$ form a basis of $T_pM$. The functions $a_i$ in \eqref{eq:vector_field_coordinates} are the \textit{component functions} of $X$ with respect to the chart $\varphi$. A vector field $X$ on $M$ is \textit{differentiable} if its component functions with respect to any  chart are differentiable functions. We will denote the set of differentiable vector fields on $M$ by $\mathcal{X}(M)$.
 A \textit{vector field $V$ along a differentiable curve $\gamma\colon (a,b)\to M$} is a differentiable function that associates to every $t\in (a,b)$ a tangent vector $V(t)\in T_pM$ such that, for any differentiable function $f\colon M\to \R$, the function $t\mapsto V(t)(f)$ is differentiable. The vector field $d\gamma \left(\frac{d}{dt}\right)$, which we will denote by $\dot{\gamma}$, is the \textit{tangent vector field} of $\gamma$. It is important to remark that a vector field along $\gamma$ may not necessarily extend to a vector field on an open subset of $M$.

\subsection{Riemannian manifolds and basic Riemannian objects}

% DEF: RIEMANNIAN MANIFOLD

\begin{definition}[Riemannian manifold]
A \textit{Riemannian metric} on a differentiable $n$-dimensional manifold $M$ is a correspondence which associates to each point $p\in M$ an inner product $\langle\cdot,\cdot \rangle_p$ (i.e.\ a symmetric, bilinear, positive-definite form) on the tangent space $T_pM$, which varies differentiably with respect to $p$ in the following sense: 
If $\varphi\colon U\subset M\to \tilde{U}\subset \R^n$ is a coordinate chart around $p$ with $\varphi = (x_1,\ldots,x_n)$, and $\frac{\partial}{\partial x_i}(q) = d\varphi^{-1}_qe_i$ for $q\in U$, then the real-valued functions 
\begin{align}
\label{eq:g_ij}
g_{ij}(x_1,\ldots,x_n) =
\left\langle \frac{\partial}{\partial x_i}(q),\frac{\partial}{\partial x_j}(q) \right\rangle_q 
\end{align}
are differentiable on $\tilde{U}$. 
A differentiable manifold with a given Riemannian metric is a \textit{Riemannian manifold}.  
\end{definition}

It is a basic fact that every differentiable manifold admits a Riemannian metric (see, for example, \cite[Chapter 2, Proposition 2.10]{doCarmo1992}).

\begin{definition}[Riemannian isometry]
A diffeomorphism $f\colon M\to N$ between Riemannian manifolds is a \textit{(Riemannian) isometry} if 
\begin{align*}
    \langle u,v\rangle_p = \langle df_pu,df_pv \rangle_{f(p)} \text{ for all } p\in M,\ u,v\in T_pM.
\end{align*}
\end{definition}

Two Riemannian manifolds are \emph{isometric} if there exists an isometry between them. From the point of view of Riemannian geometry, isometric Riemannian manifolds are equivalent objects.

A Riemannian metric allows us to measure angles and lengths of tangent vectors. This in turn makes it possible to measure lengths of piecewise differentiable curves joining two points by integrating the length of the tangent vectors to the curve along its domain. A fundamental observation  is that one may induce a metric space structure on a Riemannian manifold $M$ by defining the distance $d(p,q)$ between any two points as
\[
    d(p,q) = \inf \{\,\mathrm{Length}(\gamma) \mid \gamma \text{ is a piecewise differentiable curve joining $p$ and $q$}\,\}.
\]
Thus, the pair $(M,d)$ is a metric space and the topology induced on $M$ by the distance function $d$ coincides with the manifold topology. We say that a Riemannian manifold is \textit{complete} if it is complete as a metric space space, i.e.\ if every Cauchy sequence converges.

\begin{definition}[Metric isometry]
A diffeomorphism $f\colon M\to N$ between Riemannian manifolds is a \textit{(metric) isometry} if 
\begin{align*}
    d(p,q) = d(f(p),f(q)) \text{ for all } p,q\in M.
\end{align*}
\end{definition}

A diffeomorphism between Riemannian manifolds is a Riemannian isometry if and only if it is a metric isometry (see, for example, \cite{petersen2016}). Thus, we may simply speak of isometries of Riemannian manifolds. The set of isometries of a given Riemannian manifolds is a \textit{Lie group}, i.e. a group that is also a differentiable manifold for which the group operation and taking inverses are differentiable maps. Moreover, if the manifold is compact, its group of isometries is also compact. \\

\begin{comment}

\textbf{Add definition of a homogeneous and isotropic Riemannian manifold.}

\end{comment}
% \textcolor{red}{
% \hrule
% \hrule
% \vspace{.2cm}
% \noindent 20.12.2021 [FGG]
% \vspace{.2cm}
% }

% CF. SAKAI

Note that, in contrast to Euclidean space, there is no natural identification between tangent spaces at different points of a given differentiable manifold $M$. One is thus confronted with the problem of how to  differentiate vector fields. One may surmount this obstacle by considering a  \textit{linear connection} on $M$, i.e.\ a  map 
$\nabla\colon \mathcal{X}(M)\times \mathcal{X}(M) \to \mathcal{X}(M)$, denoted by $(X,Y)\mapsto \nabla_XY$,
that satisfies the following conditions:
\begin{align*}
    \nabla_{fX+gY}Z = f\nabla_XZ + g\nabla_YZ,\\
    \nabla_X(Y+Z)=\nabla_XY + \nabla_XZ,\\
    \nabla_X(fY) = f\nabla_XY+X(f)Y,
\end{align*}
for vector fields $X,Y,Z\in \mathcal{X}(M)$ and differentiable functions $f,g\colon M\to \R$.
 A choice of Riemannian metric $g$ on $M$ determines a unique linear connection 
%  satisfying 
% \begin{align*}
%     \nabla_XY- \nabla_YX & = XY-YX
% \end{align*}
% and
% \begin{align*}
%     \nabla_X\langle Y,Z \rangle & = \langle\nabla_XY,Z \rangle + \langle Y,\nabla_XZ \rangle,
% \end{align*}
% for any vector fields $X,Y,Z\in \mathcal{X}(M)$.
% We call this connection
$\nabla$, called the \textit{Levi--Civita connection} of the Riemannian manifold $(M,g)$. We will henceforth assume that a Riemannian manifold is always equipped with the Levi-Civita connection $\nabla$ and will call $\nabla_XY$ the \textit{covariant derivative of $Y$ by $X$}. The value $\nabla_XY(p)\in TpM$ only depends on $X(p)$ and the values of $Y$ on a differentiable curve in $M$ which is tangent to $X(p)$. Thus, for a differentiable vector field $Y$ along a curve $\gamma\colon (a,b)\to M$, we may consider the vector field $\nabla_{\dot{\gamma}}Y$, called the \textit{covariant derivative of $Y$ along $\gamma$}. Thus, by means of the Levi--Civita connection, we may canonically differentiate vector fields in a Riemannian manifold.
    The vector field $Y$  is \textit{parallel along $\gamma$} if $\nabla_{\dot{\gamma}}Y =0$ for all $t\in (a,b)$. For any $t_o\in (a,b)$ and any tangent vector $v_o\in T_{\gamma(t_o)}M$, there exists a unique parallel vector field $V$ along $\gamma$ with $V(t_o) = v_o$. We call $V$  the \textit{parallel transport} of $V_o$ along $\gamma$. In the euclidean case, parallel transport is given by parallel translating a vector along a curve.  With these definitions in hand, we may now define the \textit{parallel transport map}, which we will use in the next section.

\begin{definition}[Parallel transport map] Let $M$ be a Riemannian manifold, let $\gamma\colon (a,b) \to M$ be a differentiable curve, and fix $t_o,t_1\in (a,b)$.  The \textit{parallel transport map from $\gamma(t_o)$ to $\gamma(t_1)$} is the map 
\begin{align*}
\PT_{\gamma(t_0)\to \gamma(t_1)}\colon T_{\gamma(t_o)} M \rightarrow T_{\gamma(t_1)} M
\end{align*}
defined by letting $\PT_{\gamma(t_0)\to\gamma(t_1)}(v) = V(t_1)$,
where $V$ is the parallel transport of $v$ along $\gamma$.
\end{definition}

The parallel transport map is a linear isometry between the vector spaces $T_{\gamma(t_o)} M \rightarrow T_{\gamma(t_1)} M$, i.e.\ it is a linear isomorphism that preserves the inner products $\langle\cdot,\cdot \rangle_{\gamma(t_o)}$ and $\langle\cdot,\cdot \rangle_{\gamma(t_1)}$  induced on each tangent space by the Riemannian metric on $M$. In this way, given points $p,q\in M$, we may identify $T_p M$ and $T_q M$ by parallel transport along a differentiable curve joining $p$ and $q$. It is important to note that that this identification  depends on the choice of curve joining $p$ and $q$, as may be seen by considering parallel transport on a round $2$-sphere in Euclidean $3$-space.  

The rather technical machinery we have recalled in the preceding paragraphs is necessary to generalise classical notions from the geometry of differentiable surfaces in three-dimensional euclidean space, where one may differentiate vector fields and parallel transport vectors on a surface by appealing to extrinsic objects in the ambient euclidean space.

A differentiable curve $\gamma\colon (a,b)\to M$ is a \textit{geodesic} if $\nabla_{\dot{\gamma}}\dot{\gamma}=0$ for all $t\in (a,b)$, i.e. if its tangent vector field is parallel.  If $[c,d]\subset (a,b)$ and $\gamma\colon[c,d]\to M$ is a geodesic, the restriction of $\gamma$ to $[c,d]$ is a \textit{geodesic segment joining $\gamma(c)$ and $\gamma(d)$.} Geodesics generalize the notion of curves with zero acceleration on surfaces in Euclidean space and locally minimise distance between its points.

% \textcolor{red}{
% \hrule
% \hrule
% \vspace{.2cm}
% \noindent 21.12.2021 [FGG]
% \vspace{.2cm}
% }

Given a point $p\in M$ and a tangent vector $v\in T_pM$, the existence and uniqueness theorem of ordinary differential equations implies that there exists a unique maximal geodesic $\gamma_v\colon(-\varepsilon,\varepsilon)\to M$ with initial point $p$ and velocity $v$, i.e.\ with $\gamma_v(0)=p$ and $\gamma'_v(0)=v$. With this fact in hand, we may  now define the \emph{exponential map}, which maps a subset of $T_pM$ to $M$.  %Informally, we can think of this map as wrapping the tangent plane onto the manifold.

\begin{definition}[Exponential map]
Let $M$ be a Riemannian manifold, fix $p\in M$, and, for each $v\in T_pM$, let $\gamma_v$ be the unique maximal geodesic with $\gamma_v(0) =p$ and $\gamma'_v(0)=v$. Let 
\[
E_p =\{\,v\in T_pM \mid \gamma_v \text{ is defined on an interval containing [0,1]} \,\}. \]
The \textit{exponential map of $M$ at $p$} is the map $\exp_p\colon E_p\subset T_pM\to M$ given by 
$\exp_p(v) = \gamma_v(1)$.
\end{definition}

The exponential map is smooth and, for each $v\in T_pM$, the geodesic $\gamma_v$ with initial point $p$ and velocity $v$ is given by $\gamma_v(t)=\exp_p(tv)$ for all $t$ where $\gamma_v(t)$ and $\exp_p(tv)$ are defined. The Hopf--Rinow theorem asserts that,    a Riemannian manifold $M$ is complete if, and only if,  for every $p\in M$ and every $v\in T_pM$, the maximal geodesic $\gamma_v$ with initial point $p$ and velocity $v$ is defined for all $t\in \R$, i.e. $M$ is \textit{geodesically complete}. It follows that, for complete Riemannian manifolds, $\exp_p$ is defined on all of $T_pM$ for every $p\in M$.   A fundamental corollary to the Hopf--Rinow theorem is the fact that any two points in a complete Riemannian manifold may be joined by a shortest geodesic segment, i.e.\ a geodesic segment whose length is the distance between its endpoints. Note that such a curve may not be unique, as the example of geodesics joining antipodal points on a round sphere immediately shows.

%\textbf{Define curvature tensor and sectional curvature.}

% DEFINITION

% \begin{definition}[Homogeneous Riemannian manifold]
%   A Riemannian manifold $M$ is \textit{homogeneous} if, for all $p,q \in M$, there exists an isometry $\varphi \colon M \rightarrow M$ such that $\varphi(p)=q$.
% \end{definition}

To conclude this section, let us recall that one may define a natural Radon measure on an $n$-dimensional Riemannian manifold $(M,g)$ so that, in the case where $M$ is isometric to Euclidean space, this measure coincides with the Lebesgue measure (see \cite[Ch.\ II, Section 5]{sakai1998}). We will refer to this measure as the \textit{(Riemannian) volume} and denote it by $\vol$. The integral with respect to the Riemannian volume of a continuous function $f\colon M\to \R$ whose support is contained on the domain of a coordinate chart $\varphi\colon U\subset M\to \tilde{U}\subset \R^n$ with $\varphi = (x_1,\ldots,x_n)$ is given by
\[
\int_{M}f \dvol = \int_{\tilde{U}} \left(\,f\cdot \sqrt{\det(g_{ij})}\,\right)\circ\varphi^{-1}dx_1\cdots dx_n,
\]
where $(g_{ij})$ is the matrix whose entries are the coefficients of $g$ in the local coordinates $(x_1,\ldots,x_n)$ determined by $\varphi$ (see Equation \eqref{eq:g_ij}). Maps that preserve the volume are called \textit{volume-preserving} and will play a role in subsequent sections. Recall that, given a measure space $(X,\mu,\mathcal{A})$, a measurable space $(Y,\mathcal{B})$, and a measurable map $f$ between them, the \textit{push-forward} of $\mu$, denoted by $f_{\#}\mu$, is the measure on $(Y,\mathcal{B})$ given by
\[
(f_{\#}\mu)(B) = \mu(f^{-1}(B))
\]
for all $B\in \mathcal{B}$. A measurable map $h\colon X\to Y$ between measure spaces $(X,\mu,\mathcal{A})$ and $(Y,\nu,\mathcal{B})$ is \textit{measure preserving} if $h_{\#}\mu = \nu$. In particular, a continuous map $f\colon M\to N$ between Riemannian manifolds $M$ and $N$ is \textit{volume preserving} if $f_{\#}\vol_M = \vol_N$. Note that every isometry between Riemannian manifolds is a volume-preserving map, but volume-preserving diffeomorphisms are not necessarily isometries, as may be readily seen in the case of Euclidean space.

From now on, we will assume  all probability distributions on a given Riemannian manifold to be absolutely continuous with respect to the Riemannian volume.

% SECTION: WRAPPED PROBABILITY DISTRIBUTIONS

\section{Wrapped Probability Distributions}
\label{sec:wrapped_prob_distns}

%\textcolor{blue}{Check \cite{pennec2004}. This looks very related. } \\

\begin{comment}
\textcolor{red}{Suggested structure for this section:
\begin{enumerate}
    \item[S3.1] General definition of a wrapped distribution on a manifold. Write this in terms of $g \colon M \rightarrow M$ and $h \colon T_p M \rightarrow M$ (then can view specific constructions as a special case of this)\\
    This should include a discussion of sampling and calculating the pdf \\
    Mention also some choices for $g$ and $h$ (specific examples to follow in next subsection)
    \item[S3.2] Examples with Gaussian distribution \\
    S3.2.1 - Constructions with isometries and exponential map \\
    S3.2.2 - Constructions with area-preserving maps (instead of exponential map)
    \item[S3.3] Related work - make sure to highlight Nagano and Skopek at the very least
\end{enumerate}
}
\end{comment}

% SECTION INTRO

``Wrapping'' Euclidean probability distributions to obtain distributions on non-Euclidean spaces has a long history. We may think of ``wrapping'' a Euclidean distribution $\mu$ on a given non-Euclidean space $X$ as a procedure that assigns to $\mu$ a new distribution $w(\mu)$ on $X$, the ``wrapped'' distribution. In practice, one wishes to ``wrap'' $\mu$ in such a way that one can exercise certain control on $w(\mu)$ in terms of known properties of $\mu$. Notable examples can be found in directional statistics (see \cite[Section 3.5]{mardia1999}), such as the wrapped normal on the unit circle. This idea has also been 
considered more recently in hyperbolic space (see \cite{nagano2019, skopek2019mixed}). In this section we describe general procedures for wrapping distributions on a Riemannian manifold $N$ onto a target Riemannian manifold $M$, 
and present several examples which include some previously proposed approaches.
We will focus on distributions that arise as push-forwards. More precisely, given a differentiable map $h \colon N \rightarrow M$ and a probability distribution $\mu$ on $N$, we define a wrapped distribution $w(\mu)$ on $M$ by letting  
\begin{align}
    w(\mu) = h_{\#} \mu, \label{eq:push_forward}
\end{align}
i.e.\ we let the wrapped distribution be the push-forward of $\mu$. We call the map $h\colon N \to M $ the \textit{wrapping map} and  say that the (wrapped) distribution $w(\mu)$ is obtained by \textit{wrapping $\mu$ via $h$.} 
To use this construction in practice, it is important to consider manifolds $N$ for which probability measures are well understood and explicitly defined. We therefore restrict our attention to the case where $N$ is  Euclidean space to obtain wrapped distributions on $M$ by transforming Euclidean distributions.
We will first suppose that the wrapping map is a diffeomorphism. Note that this  simple assumption already contains the case in which $M$ is a hyperbolic space, which we will discuss in detail. We will then discuss the case in which the wrapping map is volume-preserving and present specific constructions of wrapped Euclidean distributions on round spheres. 

Note that, from an abstract point of view, we may define a wrapped probability distribution on a measurable space $(Y,\mathcal{B})$ as the push forward $h_\#\mu$ of a probability measure $\mu$ on a measure space $(X,\mathcal{A}, \mu)$ by a measurable map $h\colon(X,\mathcal{A},\mu) \to (Y,\mathcal{B})$. This allows, for example, for the consideration of non-smooth spaces more general than Riemannian manifolds (see, for example, \cite{burago2001,gigli2020}). Since our focus is on smooth spaces, however, we will not pursue this point further in this article.

% WRAPPING VIA DIFFEOMORPHISMS

\subsection{Wrapping Euclidean distributions via diffeomorphisms.}
\label{sec:generic_wrapped_distn}

Let $M$ be a Riemannian manifold diffeomorphic to $k$-dimensional Euclidean space $\R^k$. Given a diffeomorphism $h \colon 
\R^k \rightarrow M$ and a probability distribution $\mu$ on $\R^k$ with density function $f$, we let 
$f_{w}$ denote the  density function of the  probability distribution obtained by wrapping  $\mu$  via the diffeomorphism $h$. The change of variables formula implies that
 \begin{align}
      f_{w} = \dfrac{f\circ h^{-1}}{|\det(dh^{-1})|},
      %\dfrac{f\circ h^{-1}}{\left\vert \det \left( \dfrac{d}{d z} h_{\tilde{\theta}}^{-1}(z) \right) \right\vert},
 \end{align}
 which gives an explicit expression for the density function $f_M$ in terms of the Euclidean density $f$ and the wrapping map $h$. This in turn allows for straightforward sampling from $f_w$, as described in  
 Algorithm \ref{alg:sample_generic_wrapped_distn}. Note, however, that it is not obvious which choices for the Euclidean distribution $f$ and wrapping diffeomorphism $h\colon \R^k\to M$ lead to meaningful and useful distributions on $M$. To address this, we must consider the properties of these objects concurrently with properties of the manifold $M$. As an example, we may be interested in obtaining a unimodal wrapped distribution by taking a  unimodal $f$ and a wrapping diffeomorphism which preserves this property.

\begin{algorithm}[t!]
\caption{Sampling $n$ points from the wrapped distribution $f_w$} \label{alg:sample_generic_wrapped_distn}

\textbf{Input}: Euclidean density function $f$, $k$-dimensional Riemannian Manifold $M$, diffeomorphism $h \colon \R^k \rightarrow M$ \\ 
\textbf{Output}: Sample $\{z_i\}_{i=1}^n$ from wrapped distribution $f_{w}$ \\
\For{$i \in 1,2,\dots,n$}{
 Sample $u_i \in \R^{k}$ according to the probability distribution with pdf $f$ \\
 Apply the diffeomorphism to obtain $z_i = h(u_i) \in M$}
\end{algorithm}

\begin{figure}
\begin{center}
\includegraphics[width=1 \textwidth]{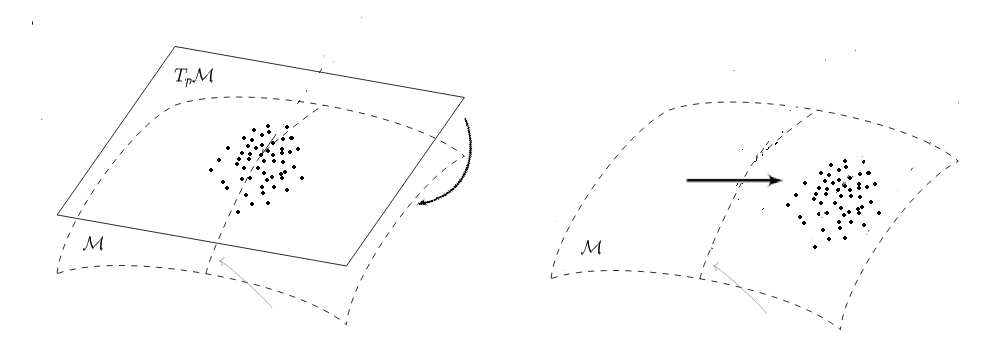}
\end{center}
\caption{Described procedure in a $2$-dimensional manifold setting. Left: Distribution in the tangent plane is attached to the manifold through an area preserving map. Right: The distribution could be transferred from a point of interest $\gamma(0)$ to another arbitrary point using an isometry of the manifold.}

%\textcolor{red}{Q: how difficult would it be to modify the notation on this figure? Not difficult. What exactly do you want me to modify.}
\end{figure}

% WRAPPED EUCLIDEAN DISTRIBUTIONS ON HYPERBOLIC SPACE

\subsubsection{Examples of wrapped Euclidean distributions on hyperbolic space}
\label{sec:wrapped_distn_hyperboloid}

 We will now discuss particular instances of wrapped distributions via diffeomorphisms by explicitly constructing wrapped Euclidean distributions on $k$-dimensional hyperbolic space $\HH^k_R$ of radius $R>0$. In this case, since $\HH^k_R$ has constant negative sectional curvature, it follows from the Cartan--Hadamard theorem (see, for example, \cite[Ch.\ 7, Theorem 3.1]{doCarmo1992}) that the  exponential map $\exp_p\colon T_p\HH^k_R \cong \R^k \to \HH^k_R$ at any point $p\in \HH^k_R$ is a diffeomorphism and provides a natural choice of wrapping map. We may generate further wrapping diffeomorphisms by composing the exponential map with an isometry of $\HH^k$,  allowing us to  ``relocate'' the wrapped distribution on hyperbolic space.

% HYPERBOLOID MODEL

%\subsubsection*{Defining the hyperboloid} 

For ease of computation, we will work with the hyperboloid model of hyperbolic space (see, for example, \cite{anderson2006} or \cite[Ch.\ 3]{lee2006}). To define this model, recall first that the \textit{Minkowski inner product} between two vectors $x,y \in \R^{k+1}$ is given by
\begin{align}
  \langle x , y \rangle = \sum_{i=1}^{k} x_i y_i - x_{k+1} y_{k+1}. \label{eq:lorenzt_inner_product}
\end{align}
Given $R>0$, %the \textit{hyperboloid model of  $k$-dimensional hyperbolic space of radius $R$}, denoted by $\HH^k_R$, is given by
let
\begin{align}
  \HH^k_R = \{ x \in \R^{k+1} \mid \langle x, x \rangle = -R^2 \mbox{ and } x_{k+1} > 0 \}.  \label{eq:hyperboloid_embedding}
\end{align}

The space $\HH^k_R$ is the upper sheet of the two-sheeted hyperboloid in $\R^{k+1}$ given by the equation 
\[
x_{k+1}^2 - \sum_{i=1}^kx_i^2 = R^2
\]
and is a $k$-dimensional (embedded) submanifold of $\R^{k+1}$. For each $p\in \HH^k_R\subset \R^{k+1}$, we may identify $T_p\HH^k_R\cong \R^k$ with a $k$-dimensional subspace of $\R^{k+1}$ whose elements are tangent vectors $\gamma'(0)$ with $\gamma\colon (-\varepsilon,\varepsilon)\to \HH^k_R\subset \R^{k+1}$ a smooth curve with $\gamma(0)=p$. For each $p\in \HH^k_R$ and each $v,w \in T_p\HH^k_R \subset \R^{k+1}$, we define
	\[
	g_p(v,w) = \langle v,w\rangle.
	\]
	That is, $g_p$ is the restriction to $T_p\HH^n$ of the Minkowski inner product on $\R^{k+1}$. One can show that the bilinear form $g_p$ on $T_p\HH^n$ is
	 positive definite and defines a Riemannian metric $g$ on $\HH^k_R$. The Riemannian manifold $(\HH^k_R,g)$ is the \textit{hyperboloid model of $k$-dimensional hyperbolic space of radius $R>0$} and has constant sectional curvature $-1/R^2<0$. We will refer to $\HH^k_1$ simply by \textit{hyperbolic space} and denote it by $\HH^k$.

\subsubsection*{Wrapped distributions after Nagano et al. \cite{nagano2019}}

As a first example, we discuss the wrapped Gaussian distribution on hyperbolic space recently proposed by Nagano et al. in \cite{nagano2019}. Although their construction was carried out for hyperbolic space $\HH^k_1$, it easily generalizes to hyperbolic space $\HH^k_R$ of arbitrary radius and we consider the general case here (cf.\ \cite{skopek2019mixed}). Setting $R=1$ yields the results in \cite{nagano2019}.

Let $p_o = (0,\ldots,0,R)\in \HH^k_R \subset \R^{k+1}$. We will refer to this point as the `vertex'' of the hyperboloid model. Fix now a second point $p\in \HH^k_R$. Our wrapping map will be the diffeomorphism given by 
\begin{align}
\label{eq:wrapping_map_nagano}
    h_p = \exp_{p} \circ\, \PT_{p_o \rightarrow p}\colon T_{p_o} \HH^k_R \rightarrow \HH^k_R,
\end{align}
where $\PT_{p_o \rightarrow p}\colon T_{p_o}\HH^k_R\to T_p\HH^k_R$ is the parallel transport map from $p_o$ to $p$ along a geodesic $\gamma_{p_{o}p}$ joining $p_0$ with $p$ and $\exp_p\colon T_p\HH^k_R\to \HH^k_R$ is the exponential map at $p$. The fact that the wrapping map $h_p$ is a well-defined diffeomorphism relies heavily on properties of negatively curved spaces. First, by the Cartan--Hadamard theorem (see \cite[Theorem 4.1]{sakai1996}), there exists a unique geodesic between every pair of points $p,q \in M$ when $M$ is a complete Riemannian manifold with non-positive sectional curvature. This ensures that $\PT_{p_o \rightarrow p}$ is uniquely defined for any point $p\in \HH^k_R$. Second, also by the Cartan--Hadamard theorem, $\exp_p\colon T_p\HH^k_R\to \HH^k_R$ is a diffeomorphism. Since $\PT_{p_o \rightarrow p}\colon T_{p_o}\HH^k_R\to T_p\HH^k_R$ is a linear isometry and, therefore, a diffeomorphism, the wrapping map $h_p=\exp_p\circ\, \PT_{p_o\to p}$ is also a diffeomorphism, being the composition of diffeomorphisms. 

Since
$T_{p_o}\HH^k_R = \{\, (x_1,\ldots,x_k,0)\in \R^{k+1}\,\}\cong \R^k$, we may use the  map $h_p$ defined in \eqref{eq:wrapping_map_nagano} to wrap a given Euclidean distribution $\mu$ onto $\HH^k_R$ while mapping the origin of $T_{p_o}\HH^k_R
\cong \R^k$ to $p\in \HH^k_R$. This is of interest, for example, when the origin in Euclidean space is the mean of $\mu$. The choice of $p_o\in \HH^k_R$ makes it straightforward to consider samples $u_i$ from the Euclidean distribution $\mu$ as points in $T_{p_o}\HH^k_R$ by appending a zero so that we have $v_i = (u_i,0)\in T_{p_o}\HH^k_R$. We may then obtain a sample on $\HH^k_R$ by considering the points $h_p(v_i)$. Algorithm~\ref{alg:sample_nagano} describes this procedure and a visualisation is given in Figure \ref{fig:nagano_ex} for the case where $\mu$ is the Euclidean Gaussian distribution, which we presently discuss in more detail. Note that the procedure we have just described is defined, more generally, when the codomain of the wrapping map $h_p$ in \eqref{eq:wrapping_map_nagano} is a Riemannian manifold which satisfies the hypotheses of the Cartan--Hadamard theorem. Such spaces are known as \textit{Cartan--Hadamard manifolds}.

Following \cite{nagano2019}, let $\nu$ be the Euclidean Gaussian distribution with zero mean on $\R^k\cong T_{p_o}\HH^k_R$ and density function $N(\cdot \mmid 0, \Sigma)$. We then wrap it via the diffeomorphism $h_p$ to obtain a wrapped normal distribution $w(\nu)$ on $\HH^k_R$ given by $w(\nu)=(h_p)_{\#}v$. We may interpret $p\in \HH^k_R$ as the ``location'' of the wrapped normal.

%{\color{blue}\textbf{Check formulas for hyperbolic space of arbitrary radius.}

The change of variables formula implies that the density function of the wrapped distribution $(h_p)_{\#}\nu$  is given by
\begin{align}
  f_{w}( z ) = \dfrac{1}{ (2 \pi)^{k/2} | \Sigma |^{-1/2} } \exp \left(-\frac {1}{2} u^T \Sigma^{-1} u \right)  \left( \dfrac{ \| u \| }{ \sinh \| u \| } \right)^{k-1}, \label{eq:nagano_pdf}
\end{align}
where $u = [(\exp_p \circ \mbox{PT}_{(0,0,\dots,0,1) \rightarrow p})^{-1}(z)]_{-(k+1)}$ and $\| u \| = \sqrt{ \langle u, u \rangle }$ for the inner-product  defined in \eqref{eq:lorenzt_inner_product}. The notation $x_{-(a)}$ represents the vector obtained from $x$ by removing its $a^{th}$ element. We refer to \cite{nagano2019,skopek2019mixed} for the derivation of this expression.

We note here that there also exist analogous constructions in other models of hyperbolic geometry, such as the Poincar\'{e} disk, and the motivation for deriving these expressions for the hyperboloid is purely one of computational convenience. %Details for the hyperboloid with general negative curvature $K$ have also recently been presented in \cite{skopek2019mixed}.

\begin{algorithm}[t!]
\caption{Sampling $n$ points from the wrapped hyperbolic Gaussian of \cite{nagano2019}} \label{alg:sample_nagano}

\textbf{Input}: Gaussian $N_{k-1}(\cdot | 0, \Sigma)$, hyperbolic mean $p \in \HH_1^k$, $PT_{p_0 \rightarrow p} \colon T_{p_0} \HH^k \rightarrow T_p \HH_1^k$ and $\exp_p \colon T_p \HH_1^k \rightarrow \HH_1^k$ \\
\textbf{Output}: Sample $\{z_i\}_{i=1}^n$ from wrapped Hyperbolic Gaussian $N_{\HH_1^k}(\cdot | \mu, \Sigma)$ \\
\For{$i \in 1,2,\dots,n$}{
Sample $u_i \in \R^{k-1}$ according to $N_{k-1}(\cdot | 0, \Sigma)$ \\
Associate each $u_i$ with a point $v_i \in T_{(0,0,\dots,0,1)} M$ by taking $v_i = [u_i, 0]$ \\
Transform $v_i$ onto $z_i \in \HH^{k}$ by taking $z_i = \left(\exp_{p} \circ \mbox{PT}_{(0,0,\dots,0,1) \rightarrow p}\right)(v_i)$
}
\end{algorithm}

\begin{figure}[t]
\begin{center}
\includegraphics[width=1 \textwidth]{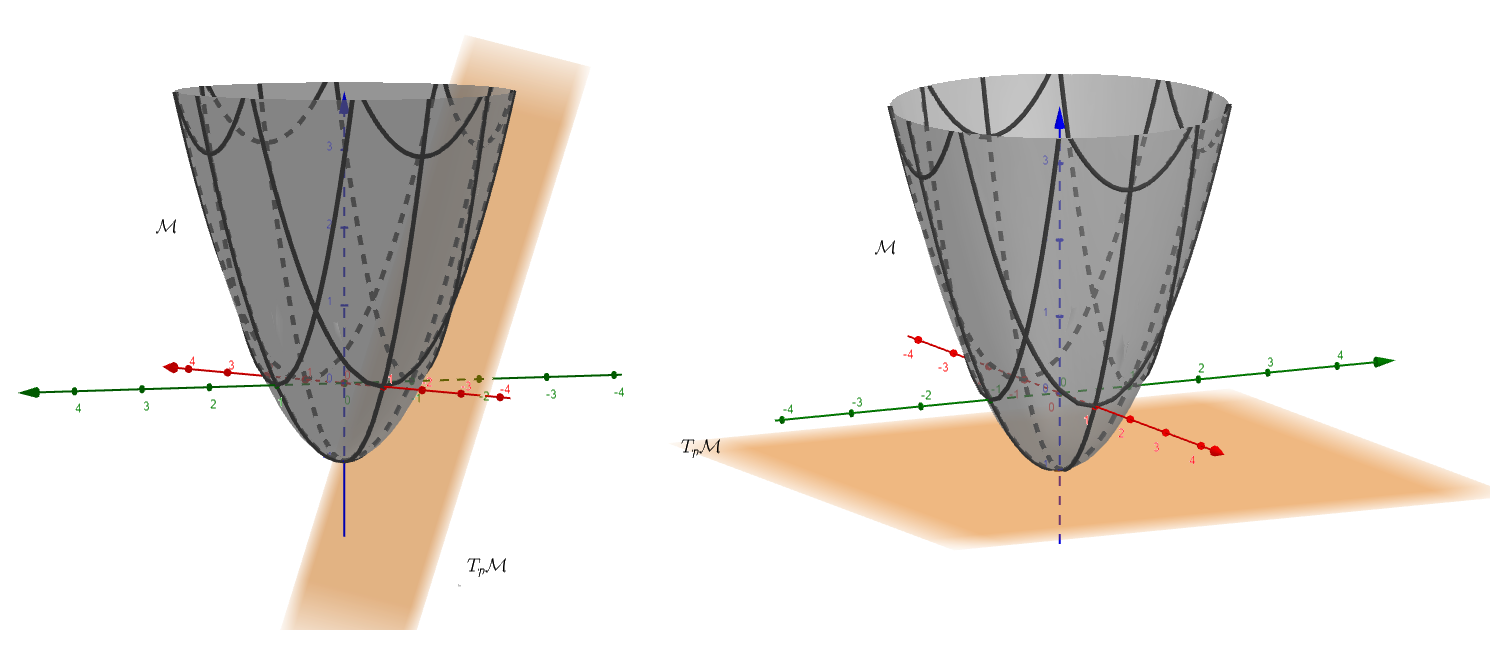}
\end{center}
\caption{Left: Way to apply the exponential map, such as in \cite{nagano2019} to obtain $z_i \in \HH_1^k$.  Right: Way to apply the Lambert map \L_p to obtain $z_i \in \HH_1^k$. In both cases, tangent plane is wrapped around the hyperboloid. } \label{fig:nagano_ex}
\end{figure}

\subsubsection*{Modifications to the wrapped Gaussian of \cite{nagano2019}}

The wrapped Gaussian construction in the previous subsection relies on parallel transport and the exponential map. Whilst these are natural choices, there also exist alternatives which may be explored for hyperbolic space.

Firs, since $\HH^k_R$ is homogeneous, it is possible to consider applying isometries in place of parallel transport. Recall from Section~\ref{sec:background} that an isometry $\varphi \colon \HH_R^k \rightarrow \HH_R^k$ is a distance-preserving map and, for the hyperboloid model, the isometry $\varphi$ which satisfies $\varphi(p_0)=p$ for $p_0 = (0,0,\dots,0,1)$ is given by the matrix
\begin{align}
    \varphi_{p_0 \rightarrow p} =  \left[ e_1 | e_2 | \dots | e_d | p \right], \label{eq:isom_hyperboloid}
\end{align}
where $e_i$ is the vector with $i^{th}$ element equal to 1 and all other elements equal to 0. Given an explicit expression for the isometries, we can instead obtain a wrapped distribution by applying the diffeomorphism 
\begin{align}
    h_{\tilde{\theta}}(v) = \left(\varphi_{p_0 \rightarrow p} \circ \exp_{p_0} \right)(v) \colon T_{p_0} \HH_1^k \rightarrow \HH_1^k, \label{eq:isom_exp_hyperboloid}
\end{align}
where $\tilde{\theta} = (p_0, p)$.
 
It is straightforward to adapt the procedure in Algorithm \ref{alg:sample_nagano} to obtain samples from the wrapped distribution defined by the diffeomorphism \eqref{eq:isom_exp_hyperboloid}. Furthermore, since the the isometry \eqref{eq:isom_hyperboloid} has determinant one, the probability density function of this distribution is the same as in \eqref{eq:nagano_pdf}. This follows since parallel transport also has a determinant equal to one, so that the Jacobian in \eqref{eq:nagano_pdf} corresponds only to the exponential map (see \cite{nagano2019} for details).

We note here that isometries and the exponential map commute in hyperbolic space (see \cite[Proposition 5.9]{lee2006}). Therefore, we may equivalently apply the diffeomorphism 
\begin{align}
    h_{\tilde{\theta}}(v) = \left(\exp_p \circ \colon d \varphi_{p_0 \rightarrow p} \right)(v) \colon T_{p_0} \HH_1^k \rightarrow \HH_1^k,
\end{align}
where $d \varphi_{p_0 \rightarrow p} \colon T_{p_0} \HH^{k} \rightarrow T_{p} \HH^{k}$ is the mapping between tangent planes induced by $\varphi_{p_0 \rightarrow p}$. The ordering in \eqref{eq:isom_exp_hyperboloid} is chosen for computational convenience and we note that we cannot change the order of the construction given in the previous section.

\subsection{Wrapping Euclidean distributions via volume-preserving maps}

%\textcolor{red}{[Advantages of wrapping via volume-preserving maps? Distribution that offers analytical results?]}

In contrast with \cite{nagano2019}, where parameters are not associated directly with mean and variance, this family of well behaved distribution are easy to handle in statistical analysis because we can have control over their distributions, associate them with other distributions (i.e. through CLT) and we can sample easily from them. More specifically, the Wrapped normal Gaussian Distribution we present is a symmetric and potentially unimodal distribution. It is a type of continuous probability distribution which is symmetric to the mean and the second parameter represents the variance analogue of normal gaussian distribution. The majority of the observations cluster around the central peak point. Fitting this distribution to data makes our life simpler because the association with normal distribution, for which we know its statistical properties. In short, you will be able to find software much quicker so that you can calculate them fast and quick. Moreover, as mentioned above advantages of this parametric distribution is that they give you real information regarding the population which can be controlled by its parameters. For example, in terms of variance, this parametric distribution can perform quite well when they have spread over when data happens to be different. Another benefit of this type of parametric distributions, again due to the association with the normal includes statistical power which means that it has more power than other tests. Therefore, you will be able to find an effect that is significant when one will exist truly.

\subsubsection{Wrapping on hyperbolic space via volume-preserving maps}

Secondly, as an alternative to the exponential map, we can consider volume-preserving maps from the tangent plane to the manifold. A particular example is given by the Lambert equal-area projection from $\R^2$ to $\HH_1^2$ which is given by
\begin{comment}

%\begin{align}
%    \lambert_{p_0}(u) = \left( \dfrac{u_1}{\sqrt{u_1^2 + u_2^2}} \sqrt{\dfrac{ (u_1^2 + u_2^2 +2)^2}{4}- 1}, \dfrac{u_2}{\sqrt{u_1^2 + u_2^2}} \sqrt{\dfrac{ (u_1^2 + u_2^2 +2)^2}{4}- 1}, \dfrac{1}{2} (u_1^2 + u_2^2 +2 ) \right), \label{eq:lambert_hyperboloid}    \\
\end{comment}
\begin{align}
    \lambert_{p_0}(u) = \left( K \sinh \left( 2 \sinh \left( \dfrac{ \sqrt{u_1^2 + u_2^2}}{2S} \right) \right)  \dfrac{u_1}{ \sqrt{u_1^2 + u_2^2} }, K \sinh \left( 2 \sinh \left( \dfrac{ \sqrt{u_1^2 + u_2^2}}{2S} \right) \right)  \dfrac{u_2}{ \sqrt{u_1^2 + u_2^2} }, K \cosh \left( 2 \sinh \left( \dfrac{ \sqrt{u_1^2 + u_2^2}}{2S} \right) \right) \right)
\end{align}\label{eq:lambert_hyperboloid} 
where $u=(u_1,u_2) \in \R^2$ and $\lambert(u) \in \HH_1^2$ and $S = 1 / \sqrt{K}$. \eqref{eq:lambert_hyperboloid} is constructed at $p_0 = (0,0,1)$ by mapping disks in $\R^2$ to disks in $\HH_1^2$ so that their area is preserved. Note for this map that $(0,0) \mapsto p_0$. 

As in the construction of \cite{nagano2019}, it is straightforward to connect a point $u$ in $\R^2$ with $v \in T_{p_0} \HH_1^2$ by taking $v = [u, 0]$. We can therefore view \eqref{eq:lambert_hyperboloid} as a map from $T_{p_0} \HH_1^k \rightarrow \HH_1^k$ and, using this, we propose the diffeomorphism 
\begin{align}
    h_{\tilde{\theta}}(v) = \left( \varphi_{{p_0} \rightarrow p} \circ \lambert_{p_0} \right) (v) \colon T_{p_0} \HH_1^k \rightarrow \HH_1^k
\end{align}
to construct a wrapped hyperbolic distribution, where $\varphi_{p_0 \rightarrow p}$ is as defined in \eqref{eq:isom_hyperboloid} and $\tilde{\theta} = (p_0, p)$.

As for the modification of \cite{nagano2019} using isometries, it is straightforward to adapt Algorithm \ref{alg:sample_nagano} to obtain samples from this wrapped distribution. However, in contrast to \cite{nagano2019}, the density function of this wrapped Gaussian is given by
\begin{align}
  p \left( z | p, \Sigma \right) = \dfrac{1}{ (2 \pi)^{k/2} | \Sigma |^{-1/2} } \exp \left(-\frac {1}{2} u^T \Sigma^{-1} u \right), \label{eq:lambert_isom_pdf}
\end{align}
where $u = \left[ \left( \varphi_{p_0 \rightarrow p} \circ \lambert_{p_0} \right)^{-1} (z) \right]_{-(k+1)}$. This expression follows from the Lambert map bring area-preserving and therefore having determinant equal to 1.

\subsection{Examples of wrapped Euclidean distributions on the round sphere}
\label{sec:wrapped_distn_sphere}

%\textcolor{red}{Which figures should we keep? Do we need the torus example? A: Maybe as an example, mentioning isometries without further details of the map (which exists but is difficult to derive).}

In this section, we now turn our attention to the sphere. We define this in Section \ref{sec:sphere_def} and discuss wrapped Gaussians in Section \ref{sec:sphere_wrapped_gaussian}.

\subsubsection{Defining the sphere}
\label{sec:sphere_def}

The $k$-dimensional sphere expressed as an embedding in $\R^{k+1}$ is given by 
%\begin{align}
%    S^k = \{ x \in \R^{k+1} \: | \: \| x \| = 1 \}, \\
\begin{align}
 S^k = \{ x \in \R^{k+1} \: | \: \| x \| = K^2 \}, 
\end{align}
where $\| x \| = \sqrt{  \sum_{i=1}^{k+1} x_i^2 }$ for $x = (x_1, x_2, \dots, x_{k+1})$.

\subsubsection{Wrapped Gaussian on the sphere}
\label{sec:sphere_wrapped_gaussian}

% should have 1) isometries, 2) lambert projection

Spherical geometry, like hyperbolic geometry, is homogenous and, unlike hyperbolic geometry, has positive curvature. This means that we can define similar diffeomorphisms to those presented in Section \ref{sec:wrapped_distn_hyperboloid}, but care must be taken in some definitions.

In the sphere, the Lambert area-preserving map (see \cite{borradaile2003}) constructed at $p_0 = (0,0,-1)$ from $D = \{x \in \R^2 | \| x \| < \sqrt{2}\} \subset \R^2$ to $S^2$ is given by
%\begin{align}
%    \lambert_{p_0}(u) = \left( \sqrt{1 - \dfrac{u_1^2 + u_2^2}{4}} u_1, \sqrt{1 - \dfrac{u_1^2 + u_2^2}{4}} u_2, \dfrac{u_1^2 + u_2^2}{2} -1 \right), \label{eq:lambert_sphere}\\
\begin{align}
\lambert_{p_0}(u) = \left( \sqrt{1 - \dfrac{u_1^2 + u_2^2}{4 K}} u_1, \sqrt{1 - \dfrac{u_1^2 + u_2^2}{4 K}} u_2, \dfrac{u_1^2 + u_2^2}{2 K} - K \right),
\end{align}\label{eq:lambert_sphere}
where $u = (u_1, u_2) \in D$. As in the previous section, we can connect $u$ to the tangent plane by taking $v = [u,0] \in T_{p_0} \colon S^k$. Given this, we are then able to apply a procedure analagous to the one described in Section \ref{sec:wrapped_distn_hyperboloid} whereby points sampled according to a Euclidean Gaussian are associated with the tangent plane and then mapped onto the sphere via a diffeomorphism. For example, we may take
\begin{align}
    h_{\tilde{\theta}}(v) = \left( \varphi_{p_0 \rightarrow p} \circ \lambert_{p_0} \right)(v) : D \subset T_{p_0}S^2 \rightarrow S^2, \label{eq:sphere_lambert_isom}
\end{align}
where $p_0 = (0,0,-1)$, $\tilde{\theta} = (p_0, p)$, $\lambert_{p_0}$ is as defined in \eqref{eq:lambert_sphere} and $\varphi_{p_0 \rightarrow p}$ denotes an isometry of $S^2$ which satisfies $\varphi (p_0) = p$.

When applying \eqref{eq:sphere_lambert_isom}, care must be taken since a standard bivariate Gaussian lies in $\R^2$. For this procedure to produce a valid probability distribution, we should restrict this distribution to lie in $D$ so that we have
\begin{align}
    p^D(u | 0, \Sigma) = \dfrac{p(u | 0, \Sigma)}{ \int_D p(\tilde{u} | 0, \Sigma) \,d \tilde{u}}, \label{eq:truncate_gaussian}
\end{align}
where $p(u | 0, \Sigma)$ denotes the pdf of $N(0, \Sigma)$. The integral required to evaluate \eqref{eq:truncate_gaussian} is given by $1 - \exp(-1/\sigma^2)$ when $\Sigma = \sigma^2 I_2$ and can be expressed as an infinite series for the case when $\Sigma = \mbox{diag}( \sigma_1^2, \sigma_2^2)$ (see \cite{gilliland1962} for details). Intuitively, this restriction does not impact the properties of the Gaussian `too much' if the marginal variances are `small enough'. As an example, when $\Sigma = \sigma^2 I_2$ we can examine the normalising constant in \eqref{eq:truncate_gaussian} and observe that this is close to 1 when $\sigma $ is roughly less than 0.4. Finally, we also comment that similar restrictions for the exponential map may also be applied in this setting to ensure the diffeomorphism $h_{\tilde{\theta}}$ is a bijection though this does not appear to be discussed in the existing literature.

\begin{figure}
\begin{center}
\includegraphics[width=1 \textwidth]{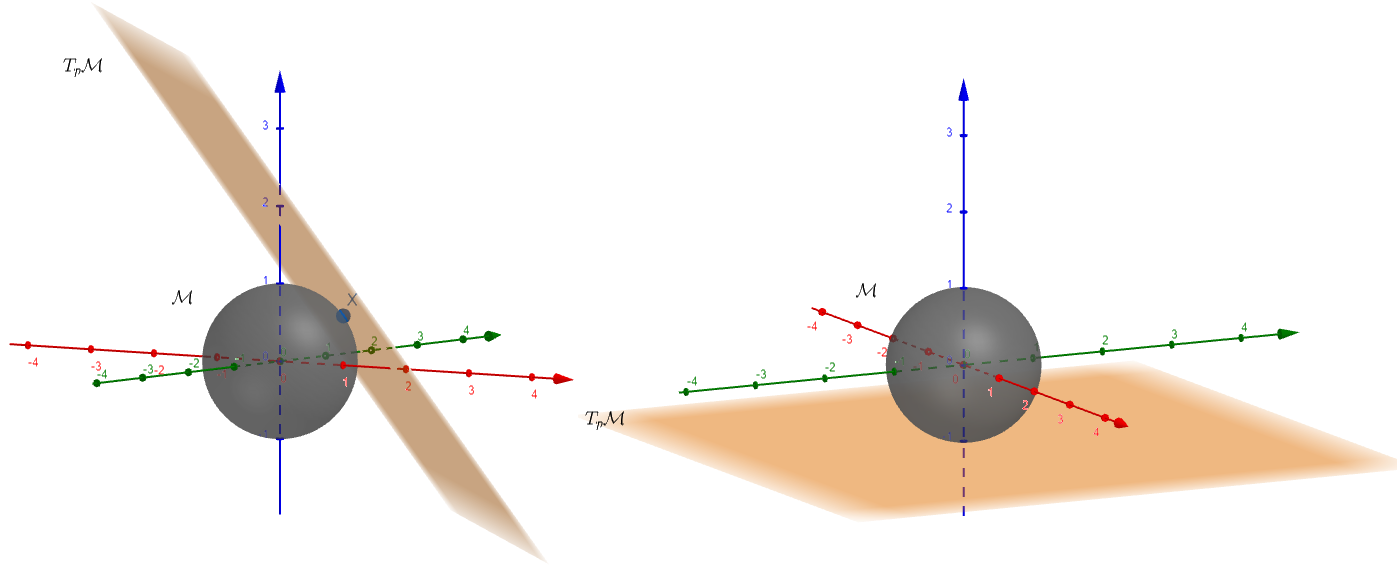}
\end{center}
\caption{Left: Way to apply the exponential map, such as in \cite{mathieu2019} and associate with $T_{(0,0,-1)} \colon S^k$.  Right: Way to apply the Lambert map \L_p to obtain $z_i \in \colon S^k$, by using as an anchor point (0,0,-1) and then composite with other isometries. In both cases, tangent plane is wrapped around the sphere.}
\end{figure}

\section{Theoretical results for manifolds with isometries}
\label{sec:theory}

The theoretical investigation that we performed in this paper answers a number of open questions and brings forth plenty of topics for future research. This includes theoretical issues, as well as questions about the possibility of applying the theoretical results in this paper in an applied statistical-topological settings. We extend the above theory in order to transfer distribution from one Riemannian manifold to another one. From example, from a $\R^n$ to $\HH^n$. The space geometry forms properties of the distributions, which a practitioner might want to control, such as symmetry in its moments and unimodality e.g. sending a distribution from a Euclidean space to a space with multiple curvatures might change the symmetry and unimodality. Sending a distribution from a space with a unique curvature to another space with unique curvature does not change those properties. More specifically we follow the procedure described below:

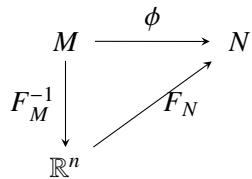
\begin{figure}[H]
  \centering
\begin{tikzpicture}
  \matrix (m) [matrix of math nodes,row sep=3em,column sep=4em,minimum width=2em]
  {
    M & N \\
    \R^n  \\};
  \path[-stealth]
    (m-1-1) edge node [left] {$F_{M}^{-1}$} (m-2-1)
            edge node [above] {$\phi$} (m-1-2)
    (m-2-1) edge node [right] {$F_{N}$} (m-1-2);
\end{tikzpicture}
\caption{Projecting distributions from one space to another, with intermediate step the euclidean space.}
\label{fig:prop5.9}
\end{figure}

, with $\phi=F_{N} \circ F^{-1}_{M}$. The unimodality arises from the 2-point homogeneous spaces and more general $n$-points homogeneous spaces. The questions that naturally arise is which properties are preserved. The following lemmas refer to the distributions constructed by our approach.\\

Riemannian spaces are endowed with geometries that are known to
be well-suited both for network data and representation learning of data
with an underlying hierarchical structure. In the next two subsections, we present theoretical results which prove the direct connection the Gaussian-like distribution on Riemannian space whose density can be evaluated analytically and differentiated with respect to the parameters. This is happening because of the nature of the mapping, which is measure preserving, from the Euclidean space to other spaces. As a consequence, we show that the properties of the Gaussian distribution in Euclidean space can be used directly to other spaces. This equips distributions in different Reimannian Spaces with analytic properties that
could never have been considered before. 

\subsection{Theoretical results}

\begin{lemma}
  Let $M$ be an $n$-dimensional complete Riemannian manifold. Fix $p \in M$ and suppose that $F_{p} \colon T_p M \rightarrow M$ is a diffeomorphism. Let $\mu$ be a probability measure on $T_p M \cong \R^n$ and let $\nu = (F_p )_{\#} \mu$ be the push-forward measure on $M$.  \\

  Let $f \colon M \rightarrow M$ be an isometry of $M$ fixing $p$ so that $f(p) = p$. \\

  If $d f_p \colon T_p M \rightarrow T_p M$ preserves the measure $\mu$ so that $( df_p)_{\#} \mu = \mu$, then the isometry $f$ preserves the measure $\nu$ so that $f_{\#} \nu = \nu$.
\label{lemma:measure_pres}  
\end{lemma}

\begin{lemma}
  Let $M$ be an $n$-dimensional Riemannian manifold. Fix $p \in M$ and suppose that $F_{p} \colon T_p M \rightarrow M$ is a diffeomorphism. Let $\mu$ be a probability measure on $T_p M \cong \R^n$ and let $\phi$ be a measure preserving diffeomorphism $\phi \colon M \rightarrow  M$ by $f	\colon F_p \circ \phi \circ F_p^{-1}$. Then $f$ preserves the push forward measure $\nu=(F_p)_{\#} \mu$ on $M$. 
\end{lemma}

\begin{lemma}
Let $M,N$ be smooth manifolds and let $\mu$ be a probability measure on $M$. If $f\colon M \rightarrow N$ is a diffeomorphism and $\phi	\colon M \rightarrow M$ is a measure-preserving diffeomorphism, then the diffeomorphism $\psi=f \circ \phi \circ f^{-1}$ preserves the push-forward measure $f_{\#} \mu$.
\end{lemma}

\begin{lemma}
Suppose we have: ($M, vol_{M}$), ($N, vol_{N})$. Let $f \colon M \rightarrow M$, $\mu$ such that $f_{\#\mu}=\mu$. Let $\lambda \colon M \rightarrow N$ with $\lambda_{\#}vol_{M}=vol_{N}$. Consider $WN_{\mu}=\lambda_{\#}\mu$. Let $\phi=\lambda \circ f \circ \lambda^{-1}$. This is a diffeomorphism in $N$. Claim $\phi \colon N \rightarrow N$ preserves $WN_{\mu}$, i.e. $\phi_{\#}(WN_{\mu})= WN_{\mu}$ Let $ B \subset N$ be measurable.
\end{lemma}

 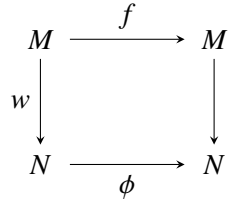
\begin{figure}[H]
  \centering
\begin{tikzpicture}
  \matrix (m) [matrix of math nodes,row sep=3em,column sep=4em,minimum width=2em]
  {
    M &  M \\
    N & {N} \\};
  \path[-stealth]
    (m-1-1) edge node [left] {$w$} (m-2-1)
            edge node [above] {$f$} (m-1-2)
            (m-2-1.east|-m-2-2) edge node [below] {$\phi$}
    node [above] {} (m-2-2)
    (m-1-2) edge node [right] {} (m-2-2);
  \end{tikzpicture}
  \caption{We want $\phi$ to preserve the $WN_{\mu}(w)$}. \label{fig:comm}
\end{figure}

\begin{lemma}
For an area-preserving mapping $f$ holds that the measure $m(f^{-1}(A))=m(A)$, where $m(\cdot)$ denotes the measure of a measurable set $A$.
\end{lemma}

\subsection{Hyperbolic and Spherical area preserving wrapped normal distribution statistical properties}

\begin{comment}

\textcolor{red}{Decide which propositions to keep - After Kathryn finish with 3 keep notation consistent}\\

\end{comment}

This subsection presents probabilistic properties of the wrapped normal distribution that facilitate inference regarding the wrapped normal distribution. The extracted results are based on the fact that the construction of the wrapped normal is based on the Normal distribution, which constitutes Wrapped Normal density function easy to compute. Therefore, estimation properties within this framework ensure that the final distribution enjoys analytical properties. The propositions are presented in three categories. The first category includes propositions (4.6-4.10) that facilitate calculations. The second category (propositions 4.11-4.15 and remark 4.16) is consisted of the distribution's properties. Finally, the last category (proposition 4.17) describes asymptotic properties of the distribution. All of them enjoy simplicity due to the normal distribution flavour of the wrapped normal distribution, in which the novelty of those properties lie.\\

The invariance properties of the Riemannian symmetric space $M$, due to symmetries, can be used to characterise Gaussian distributions on $M$. $p_0$ is the origin of the hyperboloid or a point in the sphere, in euclidean metric, and here $\sigma$ and $\Sigma$ are the variance and the covariance matrix in the manifold and $\sigma_{new}$ and $\Sigma_{new}$ in the tangent plane after the mapping.The probability of samples can be computed as in \ref{sec:generic_wrapped_distn} for volume preserving maps:
\begin{align}
    h_{\tilde{\theta}}(v) = \left( \varphi_{{p_0} \rightarrow p} \circ \lambert_{p_0} \right) (v) \colon T_{p_0} M \rightarrow M
\end{align}
to construct a wrapped hyperbolic distribution, where $\varphi_{p_0 \rightarrow p}$ is as defined in \eqref{eq:isom_hyperboloid} and $\tilde{\theta} = (p_0, p)$.

\begin{align}
  p \left( z | p, \Sigma \right) = \dfrac{1}{ (2 \pi)^{k/2} | \Sigma |^{-1/2} } \exp \left(-\frac {1}{2} u^T \Sigma^{-1} u \right), \label{eq:lambert_isom_pdf}
\end{align}
where $u = \left[ \left( \varphi_{p_0 \rightarrow p} \circ \lambert_{p_0} \right)^{-1} (z) \right]_{-(k+1)}$. We will see that the Wrapped Normal distribution converges to the Gaussian distribution $K \rightarrow 0$ and to Von Mises, Inverse stereographic Gaussian distribution in sphere as goes to infinity \cite{selvitella2019geometric}. The same happens with anisotropic hyperbolic Gaussian distribution. \\

\begin{proposition}
For the mixture distribution $p( \theta ) = \sum_{i=1}^q \pi_i WN( p_i, \Sigma_i )\
$ with $\pi_i > 0$ and $\sum_{i=1}^q \pi_i = 1$, there exists $q \in \mathbb{N}$, \
and $\theta$ such that, for an arbitrary continuous distribution $f$ and $\epsilon\
 > 0$, we have $\| f - p(\theta) \|_p < \epsilon$, where $\| \cdot \|_p$ denotes t\
he $l_p$ norm. 
\end{proposition}

\begin{comment}
{\color{red}{

\begin{proposition}
If $z_i \sim WN(p_i,\sigma_i)$, then any linear combination $u_1+u_2+\dots+u_n=\phi^K_{p_1 \rightarrow p_0}(z_1)+\phi^K_{p_2 \rightarrow p_0}(z_2)+\dots+\phi^K_{p_n \rightarrow p_0}(z_n)$ is distributed as $v\sim N( p_0, \sum_i \sigma^2_{{new}_i})$. Also, if $u$ is wrapped normal $WN( p_0, \sum_i \sigma^2_{{new}_i})$ for existing possible $\sigma_i$, then $u$ can be decomposed in wrapped normal distributions with $u_i \sim WN(p_i,\sigma_i)$.
\end{proposition}

\begin{proposition}
Suppose ${\textstyle \{X_{1},\ldots ,X_{n}\}}$ is a sequence of i.i.d. random variables with ${\textstyle \mathbb {E} [X_{i}]=p_0 }$ and ${\textstyle \operatorname {Var} [X_{i}]=\sigma ^{2}<\infty }$ following Wrapped normal distribution. Then as ${\textstyle n}$ approaches infinity, the random variables ${\textstyle {\sqrt {n}}({\bar {X}}_{n}-p_0)}$ converge in distribution to a normal ${\textstyle {\mathcal {WN}}(0,\sigma ^{2})}$:
\begin{align*}
{\displaystyle {\sqrt {n}}\left({\bar {X}}_{n}-p_0 \right)\ \xrightarrow {d} \ {\mathcal {WN}}\left(0,\sigma ^{2}\right).}
\end{align*}
\end{proposition}
}}

\end{comment}

For MLE inference the following proposition holds:
\begin{comment}

\begin{proposition}
Given data in form of a matrix $X$ of dimensions $m \times n$, if we assume that the data follows a $n$-variate Gaussian-like distribution with parameters mean $p_0$ $( n \times 1)$ and covariance matrix $\Sigma$ $(n \times n)$ the Maximum Likelihood Estimators are given by:
\begin{itemize}
\item $\hat p =  \frac{1}{m} \sum_{i=1}^m  x^{(i)}  = \bar{x}$
\item $\hat \Sigma  = \frac{1}{m} \sum_{i=1}^m (x^{(i)} - \bar{x}) (x^{(i)} - \bar{x})^T$
\end{itemize}
\end{proposition}

\end{comment}

\begin{proposition}
Let $Y \sim WN_{M} (p, \Sigma)$ and denote observations as $\underline{y} = (y_1, y_2, \dots, y_m)$, where $y_i \in M$ and $M \in \{ \mathbb{H}, \mathbb{S} \}$. The maximum likelihood estimators for $\Sigma$ is given by  
\begin{align}                              
  \hat{\Sigma} &= \dfrac{1}{m} \sum_{i=1}^m \left( h_p^{-1}(y_i) \right) \left( h_p^{-1}(y_i) \right)^T          
%  \hat{p} &= ?                 
\end{align}                                         
where $h_p^{-1} = \left( \phi_{p_0 \rightarrow p} \circ lam_{p_0} \right)^{-1}_{[-(k+1)]}$. 
\end{proposition}

For Bayesian inference, the following proposition holds:

\begin{comment}

\begin{proposition}
The conjugate priors for:
\begin{itemize}
\item Spherical case is: $p(v \mid p_0, \Sigma_{new}) = WN(z \mid  p, \Sigma)$ the priors are:
\begin{itemize}
\item Multivariate normal with known covariance matrix $\Sigma_{new}$: Multivariate normal
\item Multivariate normal with known mean $p_0$: Inverse-Wishart
\item Multivariate normal: normal-inverse-Wishart
\end{itemize}
\item Hyperboloid case is: $p(v \mid p_0, \Sigma_{new}) = WN(z \mid \ p, \Sigma)$ the prior are:
\begin{itemize}
\item Multivariate normal with known covariance matrix $\Sigma_{new}$: Multivariate normal
\item Multivariate normal with known mean $p_0$: Inverse-Wishart
\item Multivariate normal: normal-inverse-Wishart
\end{itemize}
\end{itemize}
\end{proposition}\label{prop:proposition10}  

\end{comment}

\begin{proposition}
The inverse-Wishart is a conjugate prior for $\Sigma$ when $Y \sim WN(p, \Sigma)$.\
 If $\Sigma \sim IW(\nu, \Phi)$ and we have observations $ \underline{y} = (y_1, y\
_2, \dots, y_m)$, then $\Sigma | Y \sim IW \left(\nu + m, \Phi + \sum_{i=1}^m \left(h_p^{-1}(y_i) \right) \left(h_p^{-1}(y_i) \right)^T \right)$, where $h_p^{-1} = \left( \phi_{p_0 \rightarrow p} \circ lam_{p_0} \right)^{-1}_{[-\
(k+1)]}$ 
\end{proposition}

%\textcolor{red}{Note: I don't think a conjugate prior exists for $p$}\\

The next three proposition are related with the symmetry and the unimodality of the wrapped normal distribution using an area preserving map.

\begin{proposition}
 Let a bivariate normal distribution in a tangent plane. After the area preserving mapping the resulting distribution will be symmetric.
\end{proposition}

\begin{proposition}
Let a partially monotonic distribution in a tangent plane. After the area preserving mapping the resulting distribution would be partially monotonic.
\end{proposition}

\begin{proposition}
Let a unimodal distribution in a tangent plane. After the area preserving mapping the resulting distribution would be unimodal if the curvature of the surface is constant.
\end{proposition}

The next two propositions are relate the distances of two points in the distribution after applying the map in the two spaces, $\colon E^2 \rightarrow \HH^2$ and $\colon E^2 \rightarrow \colon S^2$ respectively:

\begin{proposition}
The mapping $ \phi_{\mu_0 \rightarrow p} \circ L^K_{p_0}(x) $, where $K$ is the curvature, depends on the hyperbolic geometry and converges to $x + p_0$ as $K \rightarrow 0$. For all $x$ in the hyperboloid $\HH^n_K$ and $x \in T_{p_0} M$ , it holds that
\begin{align*}
lim_{K\rightarrow 0} L^K_{p_0}(x)=L_{p_0}(x)=x+p_0
\end{align*}
, hence the area preserving map converges to its Euclidean variant.\\
\end{proposition}

Wrapped manifold probability distributions result from the "wrapping" of the normal distribution around a manifold. The advantage of those distributions, as described rigorously above, is that sampling from Euclidean space, wrapping provides easy computation. Moreover, here, the objective is to provide proposistions, concerning monotonicity, unimodality and symmetry among those distributions in spaces where either isometries or symmetries exist.\\

\begin{proposition}
The mapping $\phi_{p_0 \rightarrow p} \circ L^K_{p_0}(x)$, where $K$ is the curvature, depends on the spherical geometry and converges to $x + \mu_0$ as $K \rightarrow 0$. For all $x$ in the Sphere $\colon S^n_K$ and $x \in T_{\mu_0}M$ , it holds that:
\begin{align*}
\lim_{K\rightarrow 0} L^K_{p_0}(x)=F_{p_0}(x)=x+p_0
\end{align*}
hence the area preserving map converges to its Euclidean variant.\\
\end{proposition}\label{prop:proposition15}  

\begin{remark}

In order to achieve bijectivity the mapping should be a diffeomorphism. For $K$ non-positive Cartan-Hadamard is used. In case $K$ is positive the distribution might be rescaled or even truncated in order to be wrapped only once around the manifold.
\end{remark}\label{rem:remark1}

\begin{comment}

\begin{proposition}
The Cartan-Hadamard theorem in conventional Riemannian geometry asserts that the universal covering space of a connected complete Riemannian manifold of non-positive sectional curvature is diffeomorphic to $\R^n$. In fact, for complete manifolds on non-positive curvature the exponential map based at any point of the manifold is a covering map.
\end{proposition}

\begin{proposition}
For the universal covering space of the tangent plane the area of the distribution is inside the open circle with radius $\pi (2R)^2$ and center the point $\mu_0$ of the tangent plane. Cut locus is not included. 
\end{proposition}

\begin{proposition}
For the universal covering space of the tangent plane the area of the distribution is inside the rectangle with one side $\pi (\sqrt{2}r)^2$ and the other less or equal to $R+r$, with $R$ the major and $r$ minor radius of the torus. center the point $\mu_0$ of the tangent plane. Cut locus of the antipodal point of $\mu_0$ in the circle of $2 \pi r$ is not included. 
\end{proposition}
\end{comment}

We compare the Wrapped Normal with the Von Mises and the Inverse Stereographic Normal \cite{selvitella2019geometric}. A classical
argument to promote the use of the Von Mises as a natural circular counterpart of the Normal Distribution is that, in the case of high-concentration limit ($\kappa \rightarrow \infty$), the two distributions resemble each others

\begin{equation*}
f_{VM}(z\mid \mu ,\kappa )\approx {\frac {1}{\sigma {\sqrt {2\pi }}}}\exp \left[{\dfrac {-(z-p )^{2}}{2\sigma ^{2}}}\right]
\end{equation*}

where $\sigma^2 = \frac{1}{\kappa}$  and the difference between the left hand side and the right hand side of the approximation converges uniformly to zero as $\kappa$  goes to infinity. Because of the theorems 3.1, 3.2 and 3.3 in \cite{selvitella2019geometric} Wrapped Normal has the same asymptotic behavior as those three distributions when the respective parameters go to infinity. Indicatively:

\begin{proposition}

Consider the two distributions
\begin{equation*}
f_{VM}(z \mid 0, \kappa) = \frac{e^{\kappa cos(z)}}{
2\pi I_0(\kappa)} 
\end{equation*}

\begin{equation*}
  f_{WN}\left( z | 0, \sigma \right) = \dfrac{1}{ (2 \pi)^{d/2} | \sigma |^{-1/2} } \exp \left(-\frac {1}{2} u^T \sigma^{-1} u \right),
\end{equation*}
,
with $0< z << \R$ and $\sigma^2 = \frac{1}{\kappa}$. Then
\begin{equation*}
\mid \mid f_{VM}(z \mid 0, \kappa) - f_{WN} (z | 0, \sigma^2) \mid \mid_{L^{\infty}([-\pi,+\pi))} \rightarrow 0, \text{ as } \kappa \rightarrow +\infty.
\end{equation*}

\end{proposition}\label{prop:proposition17}

For future work, similar asymptotics propositions involving Reimannian Normal distribution (i.e. from \cite{vonlooz2015}) and Wrapped Normal distribution, since it follows a Normal distribution after the mapping to euclidean plane, can be extracted.

\section{Experimental Study}
\label{sec:experiments}

\subsection{Network dataset}

Like in \cite{Papamichalis2021}, we illustrate these techniques using examples from an alleged benchmark dataset. The results we provide support our initial motives, which were to use inferencial models for network data. Here, the main feature of the model is that the latent positions for the nodes lie in a non-Euclidean geometry. The benchmark example is the Florentine's Family dataset, for which we assume a spherical geometry. To fit the model, we use a Metropolis-Hastings Markov Chain Monte Carlo that has as proposal a random walk in the sphere. As one of the summaries, we provide a log-likelihood convergence plot. We ran 100000 iterations of the algorithm in approximately 18 mins for each case, in R version 3.6.3. Smacof package is used for initialization, which as shown in figure \ref{fig:MCMC} helps in the convergence. Results and time (figure \ref{fig:MCMC} and table \ref{tab:alpha}) are very close to what was obtained in \cite{Papamichalis2021}.\\

%\subsubsection{Florentine Family}

\begin{figure}[H]
  \centering
  \includegraphics[width=1 \textwidth]{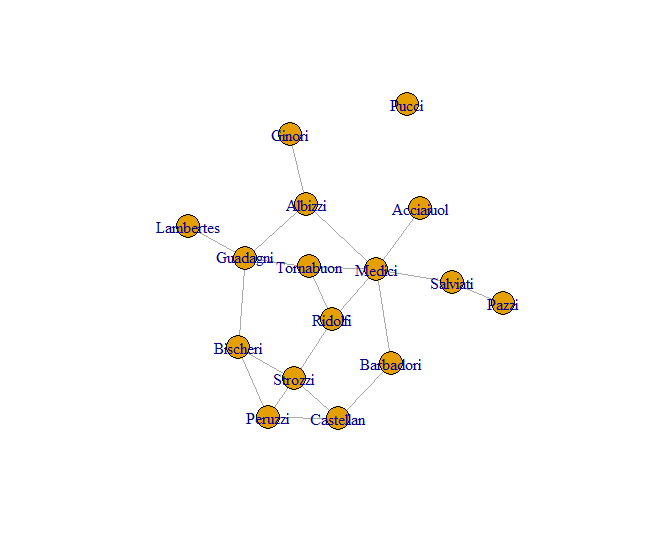}
  \caption{Network of Florentine Family with 15 non-isolated nodes. Source: igraph and netrankr, R packages}
\end{figure}

\begin{figure}[H]
  \centering
  \includegraphics[width=1 \textwidth]{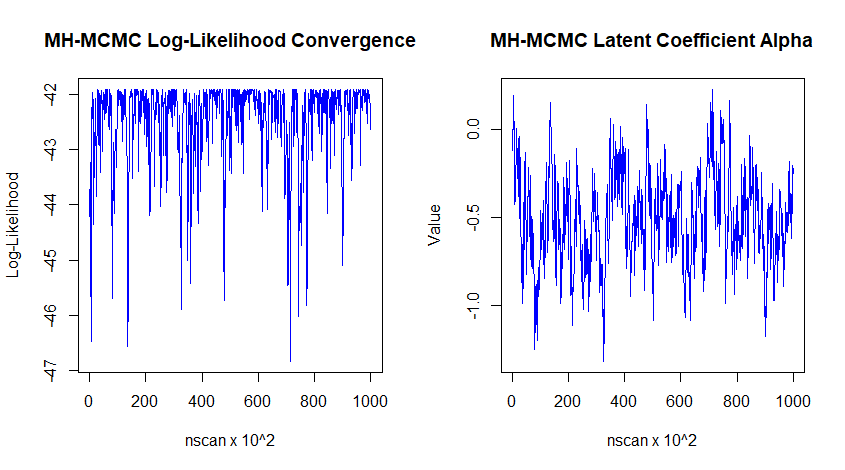}
  \caption{MH-MCMC of log-posterior for algorithm 2 convergence of \cite{Papamichalis2021}, with 15 nodes of Florentine family with area preserving Wrapped Normal distribution, for spherical geometry. For transparency a thinned version of 100 equidistant samples, among 10000, is presented. }
\end{figure}\label{fig:MCMC}

\begin{table}
\centering
\fbox{%
\begin{tabular}{| l  l  l|}
\hline
Dataset  & Nodes & MH-MCMC Estimation
 \\
\hline            
Florentine family &15   & $a=-0.587$ \\
\hline            
\end{tabular}}
\caption{Estimation of the $a$ parameter for Florentine family, which is described through spherical geometry.}
\end{table}\label{tab:alpha}

\begin{comment}

\subsubsection{Karate Club}

\begin{figure}[H]
  \centering
  \includegraphics[width=1 \textwidth]{Karate.png}
  \caption{Network of Karate Club with 34 nodes and 4 clusters. Source: igraph R package}
\end{figure}

\begin{figure}[H]
  \centering
  \includegraphics[width=1 \textwidth]{Traceplot_Karate.png}
  \caption{MH-MCMC of log-posterior for algorithm 2 convergence of \cite{Papamichalis2021},, with 34 nodes of Karate Club with area preserving Wrapped Normal distribution, using Hyperbolic geometry. For transparency a thinned version of 100 equidistant samples, among 10000, is presented. }
\end{figure}

\begin{table}
\centering

\fbox{%
\begin{tabular}{| l  l  l |}
\hline
Dataset  & Nodes & MH-MCMC Estimation 
 \\
\hline            
Karate Club &34&  $\alpha=[-0.224,-0.021]$\\
\hline
\end{tabular}}
\caption{Estimation of the $\alpha$ parameter for Karate Club through Hyperbolic geometry.}
\end{table}

The results, are very similar to the MH-MCMC results of \cite{Papamichalis2021}.

\end{comment}

In similar manner, by using a random walk in the surface of the hyperboloid, we can compare, the results of Karate Club, from \cite{Papamichalis2021} where the Poincar\'{e} disk is used to discribe the hyperbolic space instead of the hyperboloid.

\subsection{Variational Autoencoders}

The idea of autoencoders has been part of the historical landscape of neural networks for decades. Traditionally, autoencoders were used for dimensionality reduction or feature learning. Recently, theoretical connections between autoencoders andlatent variable models have brought autoencoders to the forefront of generative modeling. Autoencoders may be thought of as being a special case of feedforward networks and may be trained with all the same techniques, typically minibatch gradient descent following gradients computedby back-propagation. \\

Like autoencoders, variational autoencoders learn the parameters of a probability distribution representing the data. Since it learns to model the data, we can sample from the distribution and generate new input data samples. The goal of the variational autoencoder (VAE) is to learn a probability distribution $Pr(x)$ over a multi-dimensional variable $x$. There are two main reasons for modelling distributions. First, we might want to draw samples (generate) from the distribution to create new plausible values of $x$. Second, we might want to measure the likelihood that a new vector $x^{*}$ was created by this probability distribution. In fact, it turns out that the variational autoencoder is well-suited to the former task but not for the latter.\\

Smooth manifolds contain geometries that is known to be well-suited for representation learning of data with an underlying hierarchical structure. Among others, examples include: Hyperbolic, Euclidean and Sphrerical geometries in 2-manifold setting and the 8 geometries in 3-manifolds \cite{novello2020see}. In \cite{davidson2018hyperspherical}, the authors address this issue and propose a von Mises-Fisher (vMF) distribution instead, leading to a hyperspherical latent space, recovering hypershperical latent representations in link predictions on graphs. To this end, we propose an extentions of the theory of \cite{nagano2019wrapped}, where a wrapped hyperbolic normal distribution is constructed and examples on Synthetic Binary Tree, Atari 2600 Breakout, Word Embeddings and MNIST are presented.\\ 

The MNIST database (Modified National Institute of Standards and Technology database) is a large database of handwritten digits that is commonly used for training various image processing systems. The database is also widely used for training and testing in the field of machine learning and is used both in \cite{nagano2019wrapped,davidson2018hyperspherical}. The MNIST database contains 60,000 training images and 10,000 testing images. Half of the training set and half of the test set were taken from MNIST's training dataset, while the other half of the training set and the other half of the test set were taken from MNIST's testing dataset. \\

\begin{table}
\centering
\fbox{%
\begin{tabular}{| l  l  |}
\hline
Variational Auto-encoder & Abbreviation \\
\hline            
Normal Variational Auto-Encoder & N-SAE \\
\hline
Hyperspherical Variational Auto-Encoder &   S-VAE \\
\hline
Area Preserving Hyperspherical Variational Auto-Encoder  &   A-S-VAE \\
\hline  
 Hyperbolic Variational Auto-Encoder   &   H-VAE  \\
\hline 
Area Preserving Hyperbolic Variational Auto-Encoder    &  A-H-VAE \\
\hline    
\end{tabular}}
\caption{Abbreviations of Normal Auto-encoder, Hyperspherical Variational Auto-Encoder, Hyperbolic Variational Auto-Encoder and our methods on Spheres and Hyperbolic space, Area Preserving Hyperspherical Variational Auto-Encoder and Area Preserving Hyperbolic Variational Auto-Encoder.}
\end{table}

\begin{table}
\centering
\fbox{%
\begin{tabular}{| l  l  l   l l l|}
\hline
Dimension & N-VAE & S-VAE & A-S-VAE & H-VAE & A-H-VAE\\
\hline            
d = 2 & -135.73$\pm$.83 &-132.50$\pm$.73  & -133.76$\pm$.29&-138.61$\pm$.0.45&-137.48$\pm$.0.32\\
\hline
d = 5  &   -110.21$\pm$.21  & -108.43$\pm$.09 & -108.97$\pm$.48 &-105.38$\pm$.0.61&-105.83$\pm$.0.35\\
\hline
d = 10   &   -93.84$\pm$.30  & -93.16$\pm$.31&-93.89$\pm$.43 &-86.40$\pm$.0.28&-87.04$\pm$.0.21 \\
\hline  
d = 20   &   -88.90$\pm$.26 &  -89.02$\pm$.31 &-89.09$\pm$.87 &-79.23$\pm$.0.20&-78.98$\pm$.0.39 \\
\hline 
d = 40   &  -88.93$\pm$.30 & -90.87$\pm$.34& -90.92$\pm$.64&-78.23$\pm$.0.20&-78.89$\pm$.0.41 \\
\hline    
\end{tabular}}
\caption{Quantitative comparison of Hyperbolic VAE
against Vanilla VAE on the MNIST dataset in terms of loglikelihood (LL) for several values of latent space dimension $n$. LL was computed using 500 samples of latent variables.
We calculated the mean and the $\pm$ 1 SD with five different
experiments}
\end{table}

The LL is estimated using
importance sampling with 500 sample points \cite{burda2015importance}. The difference between the spherical case can be justified due to prior. Likewise, very easily can be realized that the results of \cite{nagano2019wrapped} are very close to our results.\\

In terms of log-likelihood (LL) all methods clearly outperformed the N-VAE in low dimensions and performs comparable to the N-VAE in higher dimensions. Empirically this shows that
the positive effect of having a uniform prior in
low dimensions both in N-VAE, S-VAE and H-VAE. In higher dimensions the spaces tend to embed all the information missing, which is captures by the nature of the autoencoder in lower dimensions.\\

We observe, that both N-VAE and S-VAE are sensitive to priors. Specifically, from \cite{davidson2018hyperspherical}, in the N-VAE setting it is observed that, the prior is too strong it will force the
posterior to match the prior shape, concentrating the samples in the center. However, this prevents the N-VAE to
correctly represent the true shape of the data and creates
instability problems for the decoder around the origin. However, as the approximate posterior differs
strongly from the prior, obtaining meaningful samples
from the latent space again becomes problematic. The S-VAE on the other hand, almost perfectly recovers
the original dataset structure, while the samples from the
approximate posterior closely match the prior distribution. However, even though the Von-Mises distribution is more informative describing the data, it is very sensitive to the prior, as well. \\

Our method A-H-VAE, either outperformed Normal and Hypersphere cases and are very close to H-VAE. with small latent dimension. As we realize, in larger dimensions the deviations between all cases decrease. In is worth menthioning that both A-S-VAE and A-H-VAE are very close to S-VAE and H-VAE, respectively. However, A-H-VAE and H-VAE outperforms all the methods due to the specific the example. The samples are described best, in low dimensions from H-VAE and A-H-VAE. The two approaches are so close that the differences that occur between are caused due to noise.

\section{Discussion}
\label{sec:discussion}

In this paper, we present a novel general framework to construct  distribution for smooth manifolds called manifold wrapped distribution, a wrapped normal distribution on smooth manifold space whose density can be evaluated analytically and differentiated with respect to the parameters. This is important because even though the default choice
of a Gaussian distribution for both the prior
and posterior represents a mathematically convenient distribution often leading to competitive results, this parameterization
fails to model data with a latent structure. Our distributions enables the gradient-based learning of the probabilistic models on smooth manifold spaces that could never have been considered before. Here, we do not intend to find the underlying
geometry of network data but instead our goal is to derive Gaussian-like distributions, by using a measure preserving maps, which describe property datasets and have useful analytical properties.  \\

Our motivation is twofold. First, the usefulness of our approach could be found in a variety of statistical and deep learning settings, such as network data and variational autoencoders. Both of them, are based in latent representations in order to describe the mechanisms that created the data. Intuitevely, properties of smooth manifolds provide a natural way to deal with those representations. Thus, a consistent and universal framework to derive distributions in those manifolds is constructed and provided to practitioners. Those distributions, could be constructed in such a way that are suitable to describe complex data sets, which euclidean space lacks to do so and at the same reduce the complexity of the calculations. Secondly, we pose several question for further investigation. The answer to those open question will help the connection between topology and statistical learning to grow. \\

For future research, we point out several important topics. Under which circumstances symmetric spaces preserve symmetries and unimodality after mapping distributions through spaces? Generally, which properties of the distribution are preserved? Which properties of the mapping can we control in order to get a distribution that we want? One idea would be to get a map given by optimizing some properties. What are good maps in terms of distributions? How to uniquely determine a dimension represent latent networks? By this, we mean a model selection criterion which keeps a balance between the number of dimensions and the information of the network. Is there a way to characterise how much the initial distribution is distorted by the exponential map, an area preserving map or generally a diffeomorphism? Is it possible to construct such frameworks for distributions on non-smooth manifold and orbifolds (or quotient spaces of Seifert $n$-manifolds? The flexibility that is provided due to their properties could be generalized by taking into account more complex objects, such as different manifolds or orbifolds in $n$-dimentional spaces, for describing more complex data. 

\printbibliography

\section*{Appendix}

\subsection{Derivation of Lambert map}

\subsubsection{On the sphere}
We consider the Lambert map defined between $S^3$ and $\mathbb{R}^2$. Let $(r, \theta)$ and $(X,Y)$ denote the polar and cartesian coordinates of $\mathbb{R}^2$, respectively. We also let $(\tilde{\phi}, \tilde{\theta})$ and $(x,y,z)$ denote the polar and cartesian coordinates of $S^3$, respectively. \\ 

In polar coordinates, the Lambert map from the plane to the sphere (see \cite{borradaile2003}) is given by     
\begin{align}                                                  
  (r, \theta)  &=  (2K \sin (\tilde{\phi}/2), \tilde{\theta}) \\    
  (\tilde{\phi}, \tilde{\theta}) &= ( 2 \sin^{-1} (r/2K), \theta)   
\end{align}                                                    
The cartesian coordinates for the plane can be written as     
\begin{align}  
  X = r \cos \theta, \hspace{1cm} Y = r \sin \theta 
\end{align}  
and for the sphere 
\begin{align}                                                  
  x = K \cos \tilde{\theta} \sin \tilde{\phi}, \hspace{.5cm} y = K \sin \tilde{\theta} \sin \tilde{\phi}, \hspace{.5cm} z = - K \cos \tilde{\phi} 
\end{align}                                                    
To derive expressions for $(x,y,z)$ in terms of $(X,Y)$ we require expressions for $\sin \tilde{\phi}, \cos \tilde{\phi}, \sin \tilde{\theta}$ and $\cos \tilde{\theta}$. 
\begin{align}                         
  X^2 + Y^2 &= 4 K^2 \sin^2 ( \tilde{\phi}/2 ) = 4 K^2 \left( \dfrac{1 - \cos \phi}{2} \right) = 2 K^2( 1- \cos \tilde{\phi}) \\                                     
  \Rightarrow - \cos \tilde{\phi} &= \dfrac{X^2 + Y^2}{2 K^2} - 1 \\               
  \sin \tilde{\phi} &= \sqrt{1 - \cos^2 \tilde{\phi} } = \sqrt{1 - \left[ \dfrac{X\
^2 + Y^2}{2 K^2} - 1 \right]^2} = \sqrt{ - \left[\dfrac{X^2 + Y^2}{2 K^2}\right]^2\
 + \dfrac{X^2 + Y^2}{K^2}  }  \\                               
            &= \sqrt{ \dfrac{X^2 + Y^2}{K^2} \left( 1 - \dfrac{X^2 + Y^2}{4 K^2} \right) }  = \sqrt{\dfrac{X^2 + Y^2}{K^2}} \sqrt{1 - \dfrac{X^2 + Y^2}{4 K^2} } \\  
  \cos \tilde{\theta} &= \cos{\theta} = \dfrac{X}{r} = \dfrac{X}{\sqrt{X^2 + Y^2}} \\          
  \sin \tilde{\theta} &= \sin{\theta} = \dfrac{Y}{r} = \dfrac{Y}{\sqrt{X^2 + Y^2}}                    
\end{align}         
This gives          
\begin{align}     
  x &= K \dfrac{X}{\sqrt{X^2 + Y^2}} \sqrt{\dfrac{X^2 + Y^2}{K^2}} \sqrt{1 - \dfrac{X^2 + Y^2}{4 K^2} }  = X \sqrt{1 - \dfrac{X^2 + Y^2}{4 K^2} } \\ 
  y &= Y \sqrt{1 - \dfrac{X^2 + Y^2}{4 K^2} } \\       
  z &= K \left( \dfrac{X^2 + Y^2}{2 K^2} - 1  \right) = \dfrac{X^2 + Y^2}{2 K} - K 
\end{align} 

\subsubsection{On the hyperboloid}

We can also define the Lambert map on the hyperboloid. Let $(X,Y)$ and $(r, \theta)$ denote cartesian and polar coordinates on $\mathbb{R}^2$, and let $(x,y,z)$ and $(\tilde{R}, \tilde{\theta})$ denote cartesian and polar coordinates on $\mathbb{H}^2$. \\

We have           
\begin{align}
  X = r \cos \theta, \hspace{1cm} Y = r \sin \theta  
\end{align}
and  
\begin{align} 
  x = K \sinh \tilde{R} \cos \tilde{\theta}, \hspace{.5cm} y = K \sinh \tilde{R} \sin \tilde{\theta}, \hspace{.5cm} z = K \cosh \tilde{R}  
\end{align} 

The Lambert map is given by  
\begin{align}  
  (r, \theta) &= (2S \sinh^{-1} (\tilde{R}/2), \tilde{\theta}) \\ 
  (\tilde{R}, \tilde{\theta} ) &= \left( 2 \sinh (r / 2S), \theta \right) 
\end{align} 
where $S = \dfrac{1}{\sqrt{K}}$. \\

We derive expressions for $(x,y,z)$ in terms of $(X,Y)$. We need expressions for $\sin \tilde{\theta}, \cos \tilde{\theta}, \cosh \tilde{R}$ and $\sinh \tilde{R}$.  
\begin{align} 
  \sqrt{X^2 + Y^2} &= r = 2S \sinh^{-1} (\tilde{R}/2) \Rightarrow \tilde{R} = 2 \sinh \left( \dfrac{ \sqrt{X^2 + Y^2}}{2S} \right)\\ 
  \cos \theta &= \cos \tilde{\theta} = \dfrac{X}{ \sqrt{X^2 + Y^2} } \\  
  \sin \theta &= \sin \tilde{\theta} = \dfrac{Y}{ \sqrt{X^2 + Y^2} }  
\end{align}
This gives 
\begin{align} 
  x &= K \sinh \left( 2 \sinh \left( \dfrac{ \sqrt{X^2 + Y^2}}{2S} \right) \right) \dfrac{X}{ \sqrt{X^2 + Y^2} } \\
  y &= K \sinh \left( 2 \sinh \left( \dfrac{ \sqrt{X^2 + Y^2}}{2S} \right) \right) \dfrac{Y}{ \sqrt{X^2 + Y^2} } \\   
  z &= K \cosh \left( 2 \sinh \left( \dfrac{ \sqrt{X^2 + Y^2}}{2S} \right) \right) 
\end{align} 

\subsection*{Determinant of Jacobian}

The Lambert maps in hyperbolic and spherical spaces has the determinant of the Jacobian equal to 1. They are area preserving. For both Hyperbolic and Spherical case we have:

\begin{align*}
\det A &= 
\begin{vmatrix}
\frac{\partial X}{\partial x} & \frac{\partial Y}{\partial x} \\ 
\frac{\partial X}{\partial y} & \frac{\partial Y}{\partial y} \notag 
\end{vmatrix}
=\begin{vmatrix}
\sqrt{\frac{2}{1+z}} & 0 \\ 
0 & \sqrt{\frac{2}{1+z}}  \notag 
\end{vmatrix}
=1
\end{align*}

, since $z$=1.

\subsection*{Exponential map and Parallel transport}

The distributions of \cite{nagano2019wrapped,skopek2019mixed} are symmetric regarding their spaces' metrics but in terms of calculations, using the euclidean metric, they lack symmetry. Their parameters are no longer the mean and the variance. Our construction allows symmetry and unimodality of the distribution in euclidean metric, so, oversimplifies the calculations and allow for analytical results both achieving symmetry in euclidean metric and the corresponding space's metric.

\subsection*{Theory}

\begin{lemma}
  Let $M$ be an $n$-dimensional complete Riemannian manifold. Fix $p \in M$ and suppose that $\exp_{p} 	\vdots T_p M \rightarrow M$ is a diffeomorphism. Let $\mu$ be a probability measure on $T_p M \cong \R^n$ and let $\nu = (\exp_p )_{\#} \mu$ be the push-forward measure on $M$. The same happens with any diffeomorphism instead of exponential map, like area preserving diffeomorphism.  \\

  Let $f \vdots M \rightarrow M$ be an isometry of $M$ fixing $p$ so that $f(p) = p$. \\

  If $d f_p \vdots T_p M \rightarrow T_p M$ preserves the measure $\mu$ so that $( df_p)_{\#} \mu = \mu$, then the isometry $f$ preserves the measure $\nu$ so that $f_{\#} \nu = \nu$.
\label{lemma:measure_pres}  
\end{lemma}

\begin{proof}
  Since $f	\vdots M \rightarrow M$ is an isometry with $f(p) = p$, $d f_p \vdots T_p M \rightarrow T_p M$ is an orthogonal transformation of $T_p M$. In other words, $df_p$ is a linear isometry of $T_p M \cong \R^n$ with the inner product $g_p$ determined by the Riemannian metric. \\

  We have the following commutative diagram
  \begin{figure}[H]
  \centering
\begin{tikzpicture}
  \matrix (m) [matrix of math nodes,row sep=3em,column sep=4em,minimum width=2em]
  {
    T_p M & T_p M \\
    M & {M} \\};
  \path[-stealth]
    (m-1-1) edge node [left] {$\exp_p$} (m-2-1)
            edge node [above] {$d f_p$} (m-1-2)
            (m-2-1.east|-m-2-2) edge node [below] {$f$}
    node [above] {} (m-2-2)
    (m-1-2) edge node [right] {$\exp_p$} (m-2-2);
  \end{tikzpicture}
  \caption{Diagram of commutative relationships}. \label{fig:comm}
\end{figure}

Then we have
\begin{align*}
  f \circ \exp_p = \exp_p \circ \vdots df_p 
\end{align*}\label{eq:f_comp_exp}

Since all of the functions in Figure \ref{fig:comm} are diffeomorphisms, it follows from \eqref{eq:f_comp_exp} that
\begin{align*}
  \exp_p^{-1} \circ f^{-1} = d f^{-1}_p \circ \exp_p^{-1}. 
\end{align*}\label{eq:f_comp_exp_inv}

Recall that $df_p$ is measure-preserving so that $\left( d f_p \right)_{\#} \mu = \mu$. Then, using \eqref{eq:f_comp_exp_inv}, we get
\begin{align*}
  \left( f_{\#} \nu \right) (A) &= \nu \left( f^{-1} (A) \right) && \mbox{ by push-forward definition} \\
                               &= \mu \left( \exp_p^{-1} \left( f^{-1} (A) \right) \right) && \mbox{ by definition of } \nu \\
                               &= \mu \left( df_p^{-1} \left( \exp_p^{-1} (A) \right) \right) && \mbox{ by } \eqref{eq:f_comp_exp_inv} \\
                               &= \left( (df_p)_{\#} \mu \right) \left( \exp_p^{-1} (A) \right) && \mbox{ by push-forward definition} \\
                               &= \mu \left( \exp_p^{-1} (A) \right) && \mbox{ since } df_p \mbox{ measure-preserving} \\
                               &= \nu(A) && \mbox{ by definition of } \nu
\end{align*}

Therefore, $f$ is measure-preserving. 

\end{proof}

\begin{remark}
  Lemma \ref{lemma:measure_pres} holds for any complete simply-connected Riemannian manifold $M$ with non-positive sectional curvature such as, for example, hyperbolic space. By the Cartan-Hadamard theorem, the exponential map at any point of such a manifold $M$ is a diffeomorphism.
\end{remark}

\begin{lemma}
  Let $M$ be an $n$-dimensional Riemannian manifold. Fix $p \in M$ and suppose that $\exp_{p} 	\vdots T_p M \rightarrow M$ is a diffeomorphism. Let $\mu$ be a probability measure on $T_p M \cong \R^n$ and let $\phi: \R^n \rightarrow \R^n$ be a measure preserving diffeomorphism. Define a diffeomorphism $\phi	\vdots M \rightarrow  M$ by $f	\vdots \exp_p \circ \phi \circ \exp_p^{-1}$. Then $f$ preserves the push forward measure $\nu=(\exp_p)_{\#} \mu$ on $M$. The same happens with any diffeomorphism, like an area preserving map.
\end{lemma}

\begin{proof}

We want to show that $f_{\#}\nu=\nu$. We compute:

\begin{align*}
f_{\#}\nu  &=\nu \circ f^{-1}&& \mbox{}\\
&=\nu(\exp_p \circ \phi^{-1} \circ \exp_p^{-1}) && \mbox{}\\
&=(\exp_p)_{\#} \mu(\exp_p \circ \phi^{-1} \circ \exp_p^{-1}) && \mbox{}\\
&=\mu(\exp_p^{-1} \circ \exp_p \circ \phi^{-1} \circ \exp_p^{-1})&& \mbox{}\\
&=\mu(\phi^{-1} \circ \exp_p^{-1})&& \mbox{}\\
&=(\phi_{\#} \mu)(\exp_p^{-1})&& \mbox{}\\
&=\mu(\exp_p^{-1})&& \mbox{}\\
&=\nu && \mbox{because } \phi \mbox{ is a measure-preserving diffeomorphism.}
\end{align*}

\end{proof}

More generally, the same proof yields:

\begin{lemma}
Let $M,N$ be smooth manifolds and let $\mu$ be a probability measure on $M$. If $f	\vdots M \rightarrow N$ is a diffeomorphism and $\phi	\vdots M \rightarrow M$ is a measure-preserving diffeomorphism, then the diffeomorphism $\psi=f \circ \phi \circ f^{-1}$ preserves the push-forward measure $f_{\#} \mu$.
\end{lemma}

\begin{proof}
Same as in the preceding lemma.
\end{proof}

The preceding proposition shows that any measure-preserving diffeomorphism on ($M,\mu)$ induces a measure-preserving diffeomorphism on ($N,f_{\#} \nu)$. Note that if $M,N$ are Riemannian manifolds and $\phi$ is also an isometry, $\psi$ may not be an isometry (example here?).

\begin{remark}
Construction of hyperbolic wrapped distribution. We construct a hyperbolic wrapped distribution $\mathcal{G}(\mu, \Sigma)$ on $n$-dimensional hyperbolic space $\HH^n$ with $\mu \in \HH^n$ and $\Sigma$ positive definite.

Nagano et al's algorithm
\begin{enumerate}
\item Sample a vector $u$ from Gaussian distribution $\colon N(0,\Sigma)$ on $\R^n$.
\item Identify $T_{\mu_0}  \cong \R^n \subset \R^{n+1}$, with $\mu_0=(1,0,\dots,0) \in \HH^n$ (where we use the upper hyperboloid model for $\HH^n$).
\item We now think of $u$ as a tangent vector in $T_{\mu_0} \HH^n$.
\item Since $\HH^n$ is simply-connected and has negative sectional curvature, there exists a unique shortest geodesic $\gamma$ joining $\mu_0$ and $\mu$.
\item Parallel transport $u$ to $\mu$ along $\gamma$ to obtain a vector $u=P_{\mu_0 \mu}(u) \in T_{\mu} \HH^n$.
\item Map $u$ to $z=\exp_{\mu}(u)$.
\end{enumerate}

In short, the construction in Nagano's et al paper is given by the following diffeomorphism:
\begin{align*}
W_{\mu}	\vdots T_{\mu_0}  \cong \R^n & \rightarrow \HH^n\\
u & \rightarrow \exp_{\mu}(P_{\mu_0 \mu}( u ))
\end{align*}
where we use the canonical identification $\R^n \cong T_{\mu_0} \subset \R^{n+1}$ induced by considering the upper hyperboloid model. We have $W_{\mu}=\exp_{\mu} \circ P_{\mu_0 \mu}$. Note that since then a unique shortest geodesic joining $\mu_0$ and $\mu$ the map $W_{\mu}$ is well-defined (The map $W_{\mu}$ is the map $proj_{\mu}$ in Nagano's paper). Thus Nagano's construction gives a canonical way of choosing a hyperbolic wrapped distribution.
\end{remark}

\begin{lemma}
Suppose ($M, vol_{M}$), ($N, vol_{N})$. Let $f 	\vdots M \rightarrow M$, $\mu$ such that $f_{\#\mu}=\mu$. Let $\lambda 	\vdots M \rightarrow N$ with $\lambda_{\#}vol_{M}=vol_{N}$. Consider $WN_{\mu}=\lambda_{\#}\mu$. Let $\phi=\lambda \circ f \circ \lambda^{-1}$. This is a diffeomorphism in $N$. Claim $\phi 	\vdots N \rightarrow N$ preserves $WN_{\mu}$, i.e. $\phi_{\#}(WN_{\mu})= WN_{\mu}$ Let $ B \subset N$ be measurable.
\end{lemma}

\begin{proof}
\begin{align*}
\phi_{\#} WN_{\mu}(B)= WN_{\mu}(\phi^{-1}(B))=WN_{\mu}((\lambda \circ f \circ \lambda^{-1})^{-1} (B))=\\
WN_{\mu}(\lambda \circ f^{-1} \circ \lambda^{-1})(B)= \lambda_{\#\mu}((\lambda \circ f \circ \lambda^{-1})^{-1} (B))=\\
\mu(\lambda^{-1} (\lambda \circ f^{-1} \circ \lambda^{-1})(B))=\mu(f^{-1} \circ \lambda^{-1}(B))=\\
f_{\#}(\mu(\lambda^{-1}(B)))=\mu(\lambda^{-1}(B))=\lambda_{\#}\mu(B)=WN_{\mu}(B).
\end{align*}
\end{proof}

 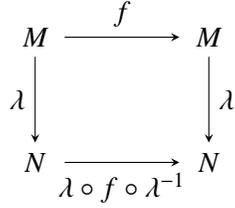
\begin{figure}[H]
  \centering
\begin{tikzpicture}
  \matrix (m) [matrix of math nodes,row sep=3em,column sep=4em,minimum width=2em]
  {
    M &  M \\
    N & {N} \\};
  \path[-stealth]
    (m-1-1) edge node [left] {$\lambda$} (m-2-1)
            edge node [above] {$f$} (m-1-2)
            (m-2-1.east|-m-2-2) edge node [below] {$\lambda \circ f \circ \lambda^{-1}$}
    node [above] {} (m-2-2)
    (m-1-2) edge node [right] {$\lambda$} (m-2-2);
  \end{tikzpicture}
  \caption{Diagram of commutative relationships. Consider $WN_{\mu}=\lambda_{\#}\mu$. Let $\phi=\lambda \circ f \circ \lambda^{-1}$. This is a diffeomorphism in $N$}. \label{fig:comm}
\end{figure}

\begin{lemma}
An area-preserving mapping $f$ if a mapping $f$ such that the measure $m(f^{-1}(A))=m(A)$, where $m(\cdot)$ denotes the measure of a measurable set $A$.
\end{lemma}

\begin{proof}
When $f$ is essentially injective then:
\begin{align*}
    \mu\bigl(f(A)\bigr)=\mu(A)\qquad\forall A\subset X\ .
\end{align*}

Now it is proven in calculus that when $f$ is essentially injective and $f(A)=B$ then for any reasonable function $g 	\vdots B\to{\mathbb R}$ one has:

\begin{align*}
    \int_B g(x)\ {\rm d}(x)=\int_Ag\bigl(f(u)\bigr)\>|J_f(u)|\>{\rm d}(u)\ .
\end{align*}

Putting $g(x) 	\vdots \equiv 1$ here gives:

\begin{align*}
    \mu\bigl(f(A)\bigr)=\mu(B)=\int_B 1\ {\rm d}(x)=\int_A |J_f(u)|\>{\rm d}(u)\ .
\end{align*}

Here the right hand side can only be $=\mu(A)$ for every $A\subset X$ if $|J_f(u)|\equiv1$.

\end{proof}

\begin{proposition}
For the mixture distribution $p( \theta ) = \sum_{i=1}^q \pi_i WN( p_i, \Sigma_i )\
$ with $\pi_i > 0$ and $\sum_{i=1}^q \pi_i = 1$, there exists $q \in \mathbb{N}$, \
and $\theta$ such that, for an arbitrary continuous distribution $f$ and $\epsilon\
 > 0$, we have $\| f - p(\theta) \|_p < \epsilon$, where $\| \cdot \|_p$ denotes t\
he $l_p$ norm. 
\end{proposition}
\begin{proof}

The same reasoning as the mixture of Gaussians.

\end{proof}

\begin{proposition}
Let $Y \sim WN_{M} (p, \Sigma)$ and denote observations as $\underline{y} = (y_1, y_2, \dots, y_m)$, where $y_i \in M$ and $M \in \{ \mathbb{H}, \mathbb{S} \}$. The maximum likelihood estimators for $p$ and $\Sigma$ are given by  
\begin{align*}                              
  \hat{\Sigma} &= \dfrac{1}{m} \sum_{i=1}^m \left( h_p^{-1}(y_i) \right) \left( h_p^{-1}(y_i) \right)^T \\          
  \hat{p} &= ?                 
\end{align*}                                         
where $h_p^{-1} = \left( \phi_{p_0 \rightarrow p} \circ lam_{p_0} \right)^{-1}_{[-(k+1)]}$. 
\end{proposition}

\begin{proof}
The log-likelihood is given by 

\begin{align*}   
L &= \log p( \underline{y}) = \log \prod_{i=1}^m p( y_i | p, \Sigma) = \sum_{i=1}^m \log p(y_i | p, \Sigma) \\ 
    &= \sum_{i=1}^m \left[ - \log \mbox{det}(2 \pi
     \Sigma)^{1/2} - \dfrac{1}{2} \left( h_p^{-1}(y_i) \right)^T \Sigma^{-1} \left( h_p^{-1}(y_i) \right) \right] \\  
  &=  - \dfrac{m}{2} \log \mbox{det}(2 \pi \Sigma) - \dfrac{1}{2} \sum_{i=1}^m \left( h_p^{-1}(y_i) \right)^T \Sigma^{-1} \left( h_p^{-1}(y_i) \right),              
\end{align*} 

where $h_p^{-1} = \left( \phi_{p_0 \rightarrow p} \circ lam_{p_0} \right)^{-1}_{[-\
(k+1)]}$. \\  
First, we determine $\hat{\Sigma}$ by finding the value which satisfies $\dfrac{\
partial L}{\partial \Sigma} = 0$ and  $\dfrac{\partial^2 L}{\partial^2 \Sigma} < 0$\
. This calculation is the same as the standard multivariate Gaussian case, and so \
we obtain  
\begin{align*} 
  \hat{\Sigma} = \dfrac{1}{m} \sum_{i=1}^m \left( h_p^{-1}(y_i) \right) \left( h_p\
^{-1}(y_i) \right)^T  
\end{align*}  
Now, we determine $\hat{p}$ in the same way. We first obtain                       
\begin{align*}  
  \dfrac{\partial L}{\partial p} &=  - \dfrac{1}{2} \dfrac{\partial}{\partial p} \sum_{i=1}^m \left( h_p^{-1}(y_i) \right)^T \Sigma^{-1} \left( h_p^{-1}(y_i) \right) \\      
                   &=  - \sum_{i=1}^m \Sigma \left( h_p^{-1}(y_i) \right) \dfrac{\partial}{\partial p} h_p^{-1}(y_i) \hspace{1cm} \mbox{ by chain rule and symmetry of } \Sigma  
\end{align*}
\end{proof}

\begin{proposition}
The inverse-Wishart is a conjugate prior for $\Sigma$ when $Y \sim WN(p, \Sigma)$.\
 If $\Sigma \sim IW(\nu, \Phi)$ and we have observations $ \underline{y} = (y_1, y\
_2, \dots, y_m)$, then $\Sigma | Y \sim IW \left(\nu + m, \Phi + \sum_{i=1}^m \left(h_p^{-1}(y_i) \right) \left(h_p^{-1}(y_i) \right)^T \right)$, where $h_p^{-1} = \left( \phi_{p_0 \rightarrow p} \circ lam_{p_0} \right)^{-1}_{[-\
(k+1)]}$ 
\end{proposition}

\begin{proof}

Follows from result for standard multivariate Gaussian with observations $X \sim N (\mu, \Sigma)$ and substituting $h_p^{-1}(y_i) $ for $(x_i - \mu)$. \\
\end{proof}

\begin{proposition}
 Let a bivariate normal distribution in a tangent plane. After the area preserving mapping the resulting distribution will be symmetric.
\end{proposition}

\begin{proof}
Without loss of generality, suppose $R=[a,b]$. This implies that $-R=[-b,-a]$. Denote $f$ to be the probability density function of the standard normal. Then
\begin{align*}
P(N \in -R) = \int_{-R}f(x)\text{ d}x = \int_{-b}^{-a}f(x)\text{ d}x\text{.}
\end{align*}
With the change of variable $y=-x$ and $dy=-dx$:
\begin{align*}
P(N \in -R) = \int_{b}^{a}f(-y)(-\text{d}y) = \int_{a}^{b}f(-y)\text{ d}y = \int_{a}^{b}f(y)\text{ d}y = P(N \in R)
\end{align*}
since $f(y) \equiv f(-y)$.
\end{proof}

\begin{proposition}
Let a partially monotonic distribution in a tangent plane. After the area preserving mapping the resulting distribution would be partially monotonic.
\end{proposition}

\begin{proof}
All paths of distributions are strictly increasing until the mean and the strictly decreasing from the mean to the other points.
\end{proof}

\begin{proposition}
Let a unimodal distribution in a tangent plane. After the area preserving mapping the resulting distribution would be unimodal if the curvature of the surface is constant.
\end{proposition}

\begin{proof}
Normal distributions preserves their unimodality as the bivariate normal distributions.
\end{proof}

\begin{proposition}
The mapping $ISO_{\mu_0 \rightarrow \mu} \circ F^K_{\mu_0}(x)$ depends on the hyperbolic geometry and converges to $x + \mu_0$ as $K \rightarrow 0$. For all $x$ in the hyperboloid $\HH^n_K$ and $x \in \mathcal{T}_{\mu_0} M$ , it holds that
\begin{align*}
lim_{K\rightarrow 0} F^K_{\mu_0}(x)=F_{\mu_0}(x)=x+\mu_0
\end{align*}
, hence the area preserving map converges to its Euclidean variant.\\
\end{proposition}

\begin{proof}
Proof is identical with \cite{skopek2019mixed}, but instead of $f=exp_{\mu}x \circ PT_{\mu_0 \rightarrow \mu}$ we use $f=ISO_{\mu \rightarrow \mu_0} \circ exp_{\mu_0}x $ for $K<0$. The same with $K>0$.
\end{proof}

\begin{proposition}
The mapping $ISO_{\mu_0 \rightarrow \mu} \circ F^K_{\mu_0}(x)$ depends on the spherical geometry and converges to $x + \mu_0$ as $K \rightarrow 0$. For all $x$ in the Sphere $\mathbb{S}^n_K$ and $x \in \mathcal{T}_{\mu_0}M$ , it holds that:
\begin{align*}
\lim_{K\rightarrow 0} F^K_{\mu_0}(x)=F_{\mu_0}(x)=x+\mu_0
\end{align*}
hence the area preserving map converges to its Euclidean variant.\\
\end{proposition}

\begin{proof}
Proof is identical with \cite{skopek2019mixed}, but instead of $f=exp_{\mu}x \circ PT_{\mu_0 \rightarrow \mu}$ we use $f=ISO_{\mu \rightarrow \mu_0} \circ exp_{\mu_0}x $ for $K<0$. The same with $K>0$.
\end{proof}

\begin{remark}

In order to achieve bijectivity the mapping should be a diffeomorphism. For $K$ non-positive Cartan-Hadamard is used. In case $K$ is positive the distribution might be rescaled or even truncated in order to be wrapped only once around the manifold.
\end{remark}

\begin{proof}

\end{proof}

\begin{proposition}

Consider the two distributions
\begin{equation*}
f_{VM}(z \mid 0, \kappa) = \frac{e^{\kappa cos(z)}}{
2\pi I_0(\kappa)} 
\end{equation*}

\begin{equation*}
  f_{WN}\left( z | 0, \sigma \right) = \dfrac{1}{ (2 \pi)^{d/2} | \sigma |^{-1/2} } \exp \left(-\frac {1}{2} u^T \sigma^{-1} u \right),
\end{equation*}
,
with $0< z << \R$ and $\sigma^2 = \frac{1}{\kappa}$. Then
\begin{equation*}
\mid \mid f_{VM}(z \mid 0, \kappa) - f_{WN} (z | 0, \sigma^2) \mid \mid_{L^{\infty}([-\pi,+\pi))} \rightarrow 0, \text{ as } \kappa \rightarrow +\infty.
\end{equation*}

\end{proposition}

\begin{proof}

Following the same reasoning as in \cite{selvitella2019geometric}.

\end{proof}

%\bibliographystyle{amsplain}
%\bibliography{gaussian_distribution_in_a_manifold}

\end{document}